\documentclass[11pt]{article}
\usepackage[utf8]{inputenc}

\usepackage[numbers,sort&compress]{natbib}

\usepackage{algorithm}

\usepackage{pgf}
\usepackage{tikz}
\usetikzlibrary{arrows, decorations.pathmorphing, backgrounds, positioning, fit, petri, automata}
\definecolor{yellow1}{rgb}{1,0.8,0.2}

\usepackage{makecell}

\usepackage{amsmath,amsthm,amsfonts,amssymb,amsbsy,amsopn,amstext}
\usepackage{graphicx,color}
\usepackage{mathbbol,mathrsfs}
\usepackage{float,ccaption}
\usepackage{indentfirst}
\usepackage{appendix}
\usepackage{subfigure}

\usepackage{emptypage}

\usepackage{booktabs}
\usepackage{threeparttable}
\usepackage{multirow}
\usepackage{diagbox}

 \newtheorem{thm}{Theorem}
 
 \newtheorem{lem}{Lemma}

 \newtheorem{ass}{Assumption}

\newcommand{\Cov}{\ensuremath{\mathrm{Cov}}} 
\newcommand{\define}{:=}
\newcommand{\A}{\mathbf{A}}
\newcommand{\Z}{\mathbf{B}}
\newcommand{\X}{\overline{c}}

\newcommand{\normm}[1]{{\left\vert\kern-0.25ex\left\vert\kern-0.25ex\left\vert #1
		\right\vert\kern-0.25ex\right\vert\kern-0.25ex\right\vert}}

\usepackage{geometry}
\begin{document}
\title{Distributed Stochastic Compositional Optimization Problems over Directed Networks}

\author{Shengchao Zhao and Yongchao Liu\thanks{School of Mathematical Sciences, Dalian University of Technology, Dalian 116024, China, e-mail: zhaoshengchao@mail.dlut.edu.cn (Shengchao Zhao), lyc@dlut.edu.cn (Yongchao Liu) }}
\date{}
\maketitle
\noindent{\bf Abstract.} We study the distributed stochastic compositional optimization problems over  directed communication networks in which
agents privately own a stochastic compositional objective function and collaborate to minimize the sum of all objective functions. We   propose a distributed
stochastic compositional gradient descent method, where the gradient tracking and the stochastic correction techniques are employed to adapt to the networks' directed structure and increase the accuracy of inner function estimation. When
the objective function
is smooth, the proposed method  achieves the  convergence rate $\mathcal{O}\left(k^{-1/2}\right)$ and sample complexity $\mathcal{O}\left(\frac{1}{\epsilon^2}\right)$ for
  finding the ($\epsilon$)-stationary point. When the objective function is strongly convex, the convergence rate is improved to $\mathcal{O}\left(k^{-1}\right)$. Moreover,
  the asymptotic normality of Polyak-Ruppert averaged iterates of the proposed method is also presented. We demonstrate the empirical performance of the proposed method  on  model-agnostic meta-learning problem and logistic regression problem.	
	
\section{Introduction}	
Stochastic compositional optimization problem  (SCO) has received extensive attention recently for its application in machine learning, stochastic programming and financial engineering, etc,  \cite{wang2017stochastic,qi2021stochastic,tianyi2021sol,guo2021randomized,wang2016acce}, which is in the form of
\begin{equation}
\label{eq:sco}
\min_{x\in\mathbb{R}^d} f(g(x)),
\end{equation}
where $f(g(x))=(f\circ g)(x)$ denotes the function composition, $f(z)\define \mathbb{E}\left[F(z;\zeta)\right],~g(x)\define\mathbb{E}\left[G(x;\phi)\right]$, $G(\cdot;\phi):\mathbb{R}^d\rightarrow \mathbb{R}^p$ and $F(\cdot;\zeta):\mathbb{R}^p\rightarrow \mathbb{R}$ are measurable functions parameterized by random variables $\zeta$ and $\phi$ respectively.

To solve the stochastic compositional optimization problem (\ref{eq:sco}), one may employ the two sample based popular schemes in stochastic optimization,  sample average approximation (SAA) and stochastic approximation (SA).
For the SAA scheme,   Dentcheva et al. \cite{Dentcheva2017}  discuss the asymptotic behavior  of the SAA problem and establish the central limit theorem for the optimal value.    Ermoliev and  Norkin \cite{Ermoliev2013}
study the  conditions for convergence in mean, almost surely of SAA problem and
provide  the large
deviation bounds for the  optimal values.
 The SA based method for SCO can be traced back to 1970s \cite{yu1976stochastic}, in which penalty functions for stochastic constraints and composite
regression models are considered. More recently,
 Wang et al. \cite{wang2017stochastic} present the stochastic compositional gradient descent  method (SCGD) for problem (\ref{eq:sco}), which is defined as follows
 \begin{equation}\label{z}
 \begin{aligned}
 &z_{k+1}=(1-\beta_k)z_k+\beta_kG(x_k;\phi_k),\\
 &x_{k+1}=x_k-\alpha_k\nabla G(x_k;\phi_k)\nabla F(z_{k+1};\zeta_k),
 \end{aligned}
 \end{equation}
where stepsize   $\alpha_k$ diminishes to zero at a faster rate than  $\beta_k$, iterates $z_{k+1}$ and $x_k$ are the estimations of inner function value $g(x_k)$ and decision variable respectively.
  For nonsmooth convex problems, the SCGD \cite{wang2017stochastic} achieves a convergence rate of $\mathcal{O}\left(k^{-1 / 4}\right)$ in the general case and $\mathcal{O}\left(k^{-2 / 3}\right)$ in the strongly convex case.
  An accelerated variant of SCGD with improved convergence
rate has been presented in \cite{wang2016acce},  where an   extrapolation-smoothing scheme is introduced.
Moreover, the variance reduction techniques, such as SVRG and SARAH, have been merged into the compositional optimization framework \cite{Huo18,Liu2021,Zhang2021multilevel}.
Note the fact that the two-timescale structure of SCGD may  decrease
the convergence rate,  Ghadimi et al. \cite{ghadimi2019single}  propose a
nested
averaged stochastic approximation method  to solve SCO (1), which is a single-timescale method    and achieves the optimal convergence rate $\mathcal{O}\left(k^{-1/2}\right)$ as   methods
for one-level unconstrained stochastic optimization.  Chen et al. \cite{tianyi2021sol} propose  the  stochastically corrected stochastic compositional gradient method (SCSC),
which  is also a single-timescale method  and achieves the optimal convergence rate  $\mathcal{O}\left(k^{-1/2}\right)$.
Specially, the SCSC read as follows:
\begin{equation*}
\begin{aligned}
&z_{k+1}=(1-\beta_k)(z_k+G(x_k;\phi_k)-G(x_{k-1};\phi_k))+\beta_kG(x_k;\phi_k),\\
&x_{k+1}=x_k-\alpha_k\nabla G(x_k;\phi_k)\nabla F(z_{k+1};\zeta_k),
\end{aligned}
\end{equation*}
 where stepsize $\alpha_k$ does not have to decay to zero at a faster rate than $\beta_k$.
Compared with SCGD, SCSC adds an extra term $G(x_k;\phi_k)-G(x_{k-1};\phi_k)$ in recursion (\ref{z}), which may reduce the tracking variance  of $g(x_k)$.
We refer \cite{Andrzej2021Multilevel,Yang2019multilevel,Zhang2021multilevel,balasubramanian2022multilevel,jiang2022multilevel} for the new developments  on  multilevel compositional optimization and \cite{Dai2017learning,Hu2020sample} on conditional stochastic optimization.


Note that  the  machine learning and financial engineering  problems tend to be characterized by large scale or distributed storage of data, consequently it is necessary to study the distributed stochastic compositional optimization problems.
  Gao and Huang \cite{gao2021fast}  first consider the distributed  stochastic compositional optimization problem (DSCO)
\begin{equation}\label{model}
\min_{x\in\mathbb{R}^d} h(x)\define\frac{1}{n}\sum_{j=1}^n f_j(g_j(x)),
\end{equation}
where $f_j(g_j(x))\define\mathbb{E}\left[F_j(\mathbb{E}\left[G_j(x;\phi_j)\right];\zeta_j)\right]$ is the local objective of agent $j$. Under the assumption that  each agent only knows its own local objective function,   Gao and Huang  \cite{gao2021fast} propose  a distributed stochastic compositional gradient descent method  for problem (\ref{model}), which is named GP-DSCGD:
\begin{align}
\label{DSC-z}&z_{i,k+1}=(1-\gamma\beta_k)z_{i,k}+\gamma\beta_kG_i(x_{i,k};\phi_{i,k}),\\
\label{DSC-x1}&\tilde{x}_{i,k+1}=\sum_{j=1}^nw_{i,j}x_{j,k}-\eta\nabla G_i(x_{i,k};\phi_{i,k})\nabla F(z_{i,k+1};\zeta_{i,k}),\\
\label{DSC-x2}&x_{i,k+1}=x_{i,k}+\beta_k(\tilde{x}_{i,k+1}-x_{i,k}),
\end{align}
where parameters $\gamma>0,\beta_k>0$, $\eta>0$ are stepszie parameters and  $W=\{w_{ij}\}$ is a symmetric and doubly stochastic matrix.
Different from the  SCGD in \cite{wang2017stochastic},     there is an additional hyperparameter $\gamma$ in  computing $z_{i,k+1}$  (\ref{DSC-z}), which  is helpful to control  the estimation variance of $z_{i,k+1}$.  On the other hand,  (\ref{DSC-x2})
 can also be beneficial to control  the estimation variance of $z_{i,k+1}$.   When
the objective function
is smooth and the communication network is undirected,   the proposed method   achieves the optimal convergence rate  $\mathcal{O}\left(k^{-1/2}\right)$.   Moreover,
  a gradient-tracking version of GP-DSCGD, named GT-DSCGD, is also proposed in \cite{gao2021fast}, where  the local gradient $\nabla G_i(x_{i,k};\phi_{i,k})\nabla F(z_{i,k+1};\zeta_{i,k})$ in (\ref{DSC-x1}) is replaced with the global gradient tracker.
  As the two methods need the   increasing batch size $\mathcal{O}(\sqrt{k})$, the corresponding  sample complexity for  finding the ($\epsilon$)-stationary point is  $\mathcal{O}(\frac{1}{\epsilon^3})$.

In this paper,  we consider the distributed stochastic  problem (\ref{model}) over directed communication networks. We propose a  gradient-tracking based distributed stochastic method, which incorporates the SCSC method \cite{tianyi2021sol} into the AB/push-pull scheme \cite{Xin2018linear,pu2018push}. The collaboration of AB scheme and SCSC induces more complex estimate errors, for example, the estimate error of inner function values is involved in the  gradient tracking process of AB scheme, the  errors of tracked gradient affect the iterations and  then   the inner function values.
Therefore, the convergence analysis techniques of AB/push-pull scheme \cite{Xin2018linear,pu2018push} and SCSC are not applicable.  Moreover, the techniques used in \cite{gao2021fast}  are not applicable as the  induced  estimate error forms
non-martingale-difference.
As far as we are concerned, the contributions of the paper can be summarized as follows.
\begin{itemize}
	\item [$\bullet$] We propose a distributed stochastic optimization method for DSCO over directed communication networks. To the best of our knowledge,
it is the first one for distributed stochastic compositional optimization problem (\ref{model}) over directed communication networks.
	(i) For the nonconvex smooth objective,
 it achieves the same order of convergence rate
$\mathcal{O}\left(k^{-1/2}\right)$ as GP-DSCGD and GT-DSCGD in \cite{gao2021fast} under constant stepsize strategy.  However, the sample complexity  for  finding the ($\epsilon$)-stationary point is  $\mathcal{O}\left(\frac{1}{\epsilon^2}\right)$  as the proposed method does not require increasing batch size in each iteration.
(ii) For the strongly convex and smooth objective, we show that the square of the distance between the iterate and the optimal solution converges to zero with rate $\mathcal{O}\left(k^{-1}\right)$ under diminishing stepsize strategy, which is
  the optimal convergence rate   as   methods
for one-level unconstrained stochastic optimization \cite{Rakhlin2012making}.

\item [$\bullet$]
We present that Polyak-Ruppert averaged iterates of the proposed method converge
in distribution to a normal random vector for any agent.  Research
on asymptotic normality results for the SA based algorithm can be traced to the works in
the 1950s \cite{chung1954stochastic,fabian1968asymptotic}. To the best of our knowledge, our result is the first asymptotic normality result for the stochastic approximation based method of DSCO. On the other hand, it is a complement to the  asymptotic  normality   on the SAA scheme for stochastic compositional optimization \cite{Dentcheva2017}.
We verify our theoretical results using two numerical examples, model-agnostic  meta-learning problem and logistic regression problem.

\end{itemize}

The rest of this paper is organized as follows. Section 2 introduces the proposed method and some standard assumptions for DSCO, communication graphs and weighted matrices. Section 3 focuses on the convergence analysis of the proposed method. At last, Section 4 presents numerical results to validate the theoretic results.

Throughout this paper, we use the following notation. $\mathbb{R}^d$ denotes  the d-dimension Euclidean space endowed with norm $\|x\|=\sqrt{\langle x,x\rangle}$.  Denote $\mathbf{1}:=(1~1\dots1)^\intercal\in \mathbb{R}^{n}$, $\mathbf{0}\define(0~0\dots0)^\intercal\in \mathbb{R}^{d}$.
$\mathbf{I}_{d}\in\mathbb{R}^{d\times d}$ stands for the identity matrix.
$\mathbf{A}\otimes \mathbf{B}$ denotes the Kronecker product of matrix $\mathbf{A}$ and $\mathbf{B}$. For any positive sequences $\{a_k\}$ and $\{b_k\}$,  $a_k=\mathcal{O}(b_k)$ if there exists $c>0$ such that $a_k\le c b_k$. The communication relationship between agents is characterized by a directed graph $\mathcal{G}=\left(\mathcal{V},\mathcal{E}\right)$, where $\mathcal{V}=\{1,2,...,n\}$ is the node set and $\mathcal{E}\subseteq\mathcal{V}\times\mathcal{V}$ is the edge set. For any $i\in\mathcal{V}$, $P_{\phi_i}$ and $P_{\zeta_i}$ are distributions of random variables $\phi_i$ and $\zeta_i$ respectively.

\section{AB-DSCSC Method}\label{part:model}

In this section, we propose a gradient tracking based distributed stochastically corrected stochastic compositional  gradient method for DSCO  over directed communication networks.



\begin{algorithm}[h]
	\caption{AB/push-pull based Distributed Stochastically Corrected Stochastic Compositional
		Gradient (AB-DSCSC): }\label{alg:AB-DSCSC}
	\textbf{Require:}  initial values $x_{i,1}\in\mathbb{R}^{d}$,  $z_{i,1}\in\mathbb{R}^{p}$, $\phi_{i,1}\stackrel{iid}\sim P_{\phi_i}$, $\zeta_{i,1}\stackrel{iid}\sim P_{\zeta_i}$, $y_{i,1}=\nabla G_i(x_{i,1};\phi_{i,1})\nabla F_i(z_{i,1}; \zeta_{i,1})$ for any $i\in\mathcal{V}$; stepsizes $\alpha_k>0$, $\beta_k>0$; nonnegative weight matrices $\mathbf{A}=\{a_{ij}\}_{1\le i,j\le n}$ and $\mathbf{B}=\{b_{ij}\}_{1\le i,j\le n}$.
	\begin{itemize}
		\item[1:]\textbf{For} $k=1,2,\cdots$ \textbf{do}
		\item [2:]\textbf{State update:} for any $i\in\mathcal{V}$,
		{\small\begin{align}
		\label{alg:x}&x_{i,k+1}=\sum_{j=1}^n a_{ij}\left( x_{j,k}-\alpha_k y_{j,k}\right).
		\end{align}}
		\item [3:]\textbf{Inner function value tracking update:} for any $i\in\mathcal{V}$,
		draw $\phi_{i,k+1}^{'}\stackrel{iid}\sim P_{\phi_i}$ to compute
		{\small\begin{equation}\label{alg:z}
		z_{i,k+1}=(1-\beta_k)\left(z_{i,k}+G_i(x_{i,k+1};\phi_{i,k+1}^{'})-G_i(x_{i,k};\phi_{i,k+1}^{'})\right)+\beta_k G_i(x_{i,k+1};\phi_{i,k+1}^{'}).
		\end{equation}}
		\item[4:]\textbf{Gradient tracking update:} for any $i\in\mathcal{V}$,
		draw $\phi_{i,k+1}\stackrel{iid}\sim P_{\phi_i},~\zeta_{i,k+1}\stackrel{iid}\sim P_{\zeta_i}$ to compute
		{\small\begin{equation}\label{alg:y}
		y_{i,k+1}=\sum_{j=1}^n b_{ij} y_{j,k}+\nabla G_i(x_{i,k+1};\phi_{i,k+1})\nabla F_i(z_{i,k+1};\zeta_{i,k+1})-\nabla G_i(x_{i,k};\phi_{i,k})\nabla F_i(z_{i,k};\zeta_{i,k}).
		\end{equation}}
		\item[4:]\textbf{end for}
	\end{itemize}
\end{algorithm}

Throughout our analysis in the paper, we make the following
two assumptions on the objective function, communication graphs and weight matrices $\mathbf{A}$ and $\mathbf{B}$.


\begin{ass}\label{ass-objective} {\rm [\textbf{Objective function}]} {\rm
Let $C_g,C_f,V_g,L_g$ and $L_f$  be positive scalars. For $\forall i\in\mathcal{V}$, $\forall x,x^{'}\in\mathbb{R}^d$, $\forall y,y^{'}\in\mathbb{R}^p$,
\begin{itemize}
	\item[(a)]  functions $G_i(\cdot;\phi_i)$ and $F_i(\cdot;\zeta_i)$ are $L_g$ and $L_f$ smooth, that is,
	\begin{equation*}
	\|\nabla G_i(x;\phi_i)-\nabla G_i(x^{'};\phi_i)\|\le L_g\|x-x^{'}\|,
	\end{equation*}
	and
	\begin{equation*}
	\|\nabla F_i(y;\zeta_i)-\nabla F_i(y^{'};\zeta_i)\|\le L_f\|y-y^{'}\|;
	\end{equation*}
    \item [(b)]
	\begin{equation*}
	\mathbb{E}\left[\nabla G_i(x;\phi_{i,1})\nabla F_i(y;\zeta_{i,1})\right]=\nabla g_i(x)\nabla f_i(y),\quad \mathbb{E}\left[G_i(x;\phi_{i,1}^{'})\right]= g_i(x);
	\end{equation*}
	\item [(c)] the stochastic gradients of $f_i$ and $g_i$ are bounded in expectation, that is
	\begin{equation*}
	\mathbb{E}\left[\|\nabla G_i(x;\phi_i)\|^2\big|\zeta_i\right]\le C_g,\quad \mathbb{E}\left[ \|\nabla F_i(y;\zeta_i)\|^2\right]\le C_f;
	\end{equation*}
	\item [(d)] function $G(\cdot;\phi_i)$ has bounded variance, i.e., $\mathbb{E}\left[\|G_i(x;\phi_i)-g_i(x)\|^2\right]\le V_g$;
\end{itemize}}
\end{ass}
Assumption \ref{ass-objective} is standard assumption for stochastic compositional optimization problem \cite{tianyi2021sol,wang2017stochastic,gao2021fast}. Conditions  (b) and (d) in Assumption \ref{ass-objective} are
analogous to the unbiasedness and bounded variance assumptions for non-compositional stochastic optimization problems.

\begin{ass}\label{ass:matrix}  {\rm[\textbf{weight matrices and networks}]} {\rm
Let $\mathcal{G}_A=\left(\mathcal{V},\mathcal{E}_\A\right)$ and $\mathcal{G}_{B^\intercal}=\left(\mathcal{V},\mathcal{E}_{\Z^\intercal}\right)$ be subgraphs of $\mathcal{G}$ induced by  matrices $\mathbf{A}$ and $\mathbf{B}^\intercal$ respectively. 
	\begin{itemize}
		\item [(a)]The matrix $\mathbf{A}\in\mathbb{R}^{n\times n}$ is nonnegative row stochastic and $\mathbf{B}\in\mathbb{R}^{n\times n}$ is nonnegative column stochastic, i.e., $\mathbf{A}\mathbf{1}=\mathbf{1}$ and $\mathbf{1}^\intercal\mathbf{B}=\mathbf{1}^\intercal$. In addition, the diagonal entries of $\mathbf{A}$ and $\mathbf{B}$ are positive, i.e., $\mathbf{A}_{ii}>0$ and $\mathbf{B}_{ii}>0$ for all $i\in \mathcal{V}$.
		\item[(b)]  The graphs $\mathcal{G}_A$ and $\mathcal{G}_{B^\intercal}$ each contain at least one spanning tree. Moreover, there exists at least one node that is a root of spanning trees for both $\mathcal{G}_A$ and $\mathcal{G}_{B^\intercal}$, i.e. $\mathcal{R}_A\cap \mathcal{R}_{B^\intercal}\ne\emptyset$, where $\mathcal{R}_A$ ( $\mathcal{R}_{B^\intercal}$) is the set of roots of all possible spanning trees in the graph $\mathcal{G}_A$ ( $\mathcal{G}_{B^\intercal}$).
	\end{itemize}}
\end{ass}
It is worth noting that Assumption \ref{ass:matrix} (b) is weaker than requiring that both $\mathcal{G}_A$ and $\mathcal{G}_{B^\intercal}$
are strongly connected, which offers greater flexibility in the design of $\mathcal{G}_A$ and $\mathcal{G}_{B}$ \cite{pu2020push}.  Under  Assumption \ref{ass:matrix}, the matrix $\mathbf{A}$ has a nonnegative left eigenvector $\mathbf{u}^\intercal$ (w.r.t. eigenvalue 1) with $\mathbf{u}^\intercal\mathbf{1}=n$, and the matrix $\mathbf{B}$ has a nonnegative right eigenvector $\mathbf{v}$ (w.r.t. eigenvalue 1) with $\mathbf{1}^T\mathbf{v}=n$. Moreover, $\mathbf{u}^\intercal \mathbf{v}>0$.

For easy of presentation, we rewrite AB-DSCSC ((\ref{alg:x})-(\ref{alg:z})) in a compact form:
\begin{equation}\label{alg:new form}
\begin{aligned}
&\mathbf{x}_{k+1}=\tilde{\mathbf{A}}\left(\mathbf{x}_k-\alpha_k\mathbf{y}_k\right),\\
&\mathbf{z}_{k+1}=(1-\beta_k)\left(\mathbf{z}_k+\mathbf{G}_{k+1}^{(1)}-\mathbf{G}_{k+1}^{(2)}\right)+\beta_k \mathbf{G}_{k+1}^{(1)},\\
&\mathbf{y}_{k+1}=\tilde{\mathbf{B}}\mathbf{y}_k+\mathbf{H}_{k+1}-\mathbf{H}_k,
\end{aligned}
\end{equation}
where $\tilde{\mathbf{A}}\define\mathbf{A}\otimes \mathbf{I}_d,~ \tilde{\mathbf{B}}\define\mathbf{B}\otimes \mathbf{I}_d$, the vectors $\mathbf{x}_k$, $\mathbf{y}_k$, $\mathbf{z}_k$, $ \mathbf{G}_{k+1}^{(1)}$, $ \mathbf{G}_{k+1}^{(2)}$ and $\mathbf{H}_k$ concatenate all $x_{i,k}$'s, $y_{i,k}$'s, $z_{i,k}$'s, $G_i(x_{i,k+1},\phi_{i,k+1}^{'})$'s, $ G_i(x_{i,k},\phi_{i,k+1}^{'})$'s and $\nabla G_i(x_{i,k};\phi_{i,k})\nabla F_i(z_{i,k};\zeta_{i,k})$'s respectively.

\section{Convergence analysis}
In this section, we derive convergence rates of AB-DSCSC. We also investigate the asymptotic normality of AB-DSCSC when the objective function  is strongly convex.  We first present a technical lemma  which provides some norms for studying the consensus of AB-DSCSC.
\begin{lem}\label{lem:norm}
Under Assumption \ref{ass:matrix}, there exist vector norms, denoted as $\|\cdot\|_\A$, $\|\cdot\|_\Z$, on $ \mathbb{R}^{nd}$ such that the corresponding induced matrix norms 
 $\normm{\mathbf{W}}_\A\define\sup_{\mathbf{x}\neq0}\frac{\|\hat{\mathbf{W}}\mathbf{x}\|_\A}{\|\mathbf{x}\|_\A},\quad\normm{\mathbf{W}}_\Z\define\sup_{\mathbf{x}\neq0}\frac{\|\hat{\mathbf{W}}\mathbf{x}\|_\Z}{\|\mathbf{x}\|_\Z}$
 for $\mathbf{W}\in\mathbb{R}^{nd\times nd}$ satisfy:
\begin{equation}
\normm{\tilde{\mathbf{A}}-\frac{\mathbf{1}\mathbf{u}^\intercal}{n}\otimes \mathbf{I}_d}_\A<1,\quad \normm{\tilde{\mathbf{B}}-\frac{\mathbf{v}\mathbf{1}^\intercal}{n}\otimes \mathbf{I}_d}_\Z<1.
\end{equation}
Additionally,
let $\|\cdot\|_*$ and $\|\cdot\|_{**}$ be any two vector norms of $\|\cdot\|$, $\|\cdot\|_\A$ or $\|\cdot\|_\Z$. There exists a constant $\X>1$ such that
\begin{equation}\label{norm-bound-1}
\|\mathbf{x}\|_*\le \X \|\mathbf{x}\|_{**},\quad \forall \mathbf{x}\in\mathbb{R}^{nd}.
\end{equation}
\end{lem}
\begin{proof}
 Under Assumption \ref{ass:matrix}, the conditions of \cite[Lemma 3]{Song2021CompressedGT} hold and then there exists an invertible matrix $\mathbf{A}_*\in \mathbb{R}^{d\times d}$ such that
\begin{equation*}
\normm{\mathbf{A}-\frac{\mathbf{1}\mathbf{u}^\intercal}{n}}_*=\normm{\A_*\left(\mathbf{A}-\frac{\mathbf{1}\mathbf{u}^\intercal}{n}\right)\A_*^{-1}}<1,
\end{equation*}
where $\normm{\cdot}_*$ and $\normm{\cdot}$ are  matrix norms induced by vector norms $\|x\|_*\define\|\mathbf{A}_*x\|$ and 2-norm respectively. Let $\hat{\mathbf{A}}=\mathbf{A}_*\otimes \mathbf{I}_d$. Noting that $\left(\mathbf{W}_1\otimes\mathbf{W}_2\right)^{-1}=\mathbf{W}_1^{-1}\otimes \mathbf{W}_2^{-1}$ for any invertible matrices $\mathbf{W}_1,\mathbf{W}_2\in \mathbb{R}^{nd\times nd}$, $\hat{\mathbf{A}}^{-1}=\mathbf{A}_*^{-1}\otimes \mathbf{I}_d$. Therefore, vector matrix $\|\mathbf{x}\|_{\mathbf{A}}\define\|\hat{\mathbf{A}}\mathbf{x}\|$ is well defined and the corresponding induced matrix norm $\normm{\cdot}_{\mathbf{A}}$ satisfies
\begin{align*}
\normm{\tilde{\mathbf{A}}-\frac{\mathbf{1}\mathbf{u}^\intercal}{n}\otimes \mathbf{I}_d}_{\mathbf{A}}&=\normm{\hat{\mathbf{A}}\left(\tilde{\mathbf{A}}-\frac{\mathbf{1}\mathbf{u}^\intercal}{n}\otimes \mathbf{I}_d\right)\hat{\mathbf{A}}^{-1}}\\
&=\normm{\left[\A_*\left(\mathbf{A}-\frac{\mathbf{1}\mathbf{u}^\intercal}{n}\right)\A_*^{-1}\right]\otimes \mathbf{I}_d}\\
&=\normm{\A_*\left(\mathbf{A}-\frac{\mathbf{1}\mathbf{u}^\intercal}{n}\right)\A_*^{-1}}<1.
\end{align*}
By the similar analysis, there exists $\hat{\mathbf{B}}$ such that
\begin{align*}
\normm{\tilde{\mathbf{B}}-\frac{\mathbf{v}\mathbf{1}^\intercal}{n}\otimes \mathbf{I}_d}_{\mathbf{B}}&=\normm{\hat{\mathbf{B}}\left(\tilde{\mathbf{B}}-\frac{\mathbf{v}\mathbf{1}^\intercal}{n}\otimes \mathbf{I}_d\right)\hat{\mathbf{B}}^{-1}}<1.
\end{align*}

The inequality (\ref{norm-bound-1}) follows from the equivalence relation of all norms on $\mathbb{R}^d$. The proof is complete.
\end{proof}

The next lemma studies the asymptotic consensus of AB-DSCSC.
\begin{lem}\label{lem:rate}
Suppose Assumptions \ref{ass-objective}-\ref{ass:matrix} hold. Stepsize $\alpha_k$ is nonincreasing and $\lim_{k\rightarrow\infty}\frac{\alpha_k}{\alpha_{k+1}}=1$. Define auxiliary sequence
$\{\mathbf{y}_k^{'}\}$ as
\begin{equation}\label{g-tra}
\begin{aligned}
&\mathbf{y}_{k+1}^{'}=\tilde{\mathbf{B}}\mathbf{y}_k^{'}+\mathbf{J}_{k+1}-\mathbf{J}_k,
\end{aligned}
\end{equation}
where vectors $\mathbf{J}_k$ and $\mathbf{y}_{1}^{'}$ concatenate all $\nabla g_i(x_{i,k})\nabla f_i(z_{i,k})$'s and $\nabla g_i(x_{i,1})\nabla f_i(z_{i,1})$'s respectively.
Then
\begin{align}
&\mathbb{E}\left[\|\mathbf{x}_{k+1}-\mathbf{1}\otimes\bar{x}_{k+1}\|_\A^2\right]+c_4\mathbb{E}\left[\|\mathbf{y}_{k+1}^{'}-\mathbf{v}\otimes\bar{y}_{k+1}^{'}\|_\Z^2\right]\notag\\
\label{ine-0}&\le\rho^{k} \left(\mathbb{E}\left[\left\|\mathbf{x}_1-\mathbf{1}\otimes \bar{x}_1\right\|_\A^2\right]+c_4\mathbb{E}\left[\left\|\mathbf{y}_{1}^{'}-\mathbf{v}\otimes \bar{y}_{1}^{'}\right\|_\Z^2\right]\right)+(c_1+c_3c_4)\sum_{t=1}^k\rho^{k-t}\alpha_t^2,
\end{align}
where $\bar{x}_k\define \left(\frac{\mathbf{u}^\intercal}{n}\otimes\mathbf{I}_{d}\right)\mathbf{x}_k$, $\bar{y}_k^{'}\define\left(\frac{\mathbf{1}^\intercal}{n}\otimes\mathbf{I}_{d}\right)\mathbf{y}_k^{'}$ and $\rho=\max\left\{\frac{1+\tau_\mathbf{B}^2}{2},~\frac{3+\tau_\mathbf{A}^2}{4}\right\}$, $c_b=\max\left\{\X,\frac{\normm{\mathbf{B}-\mathbf{I}_{n}}_\Z}{\tau_\Z}\X\right\}$,
\begin{align}
&c_1=\frac{1+\tau_\mathbf{A}^2}{1-\tau_\mathbf{A}^2}\normm{\mathbf{I}_{n}-\frac{\mathbf{1}\mathbf{u}^\intercal}{n}}_\A^2\X^2\frac{c_b^2nC_gC_f}{(1-\tau_\Z)^2}\notag\\
&c_2=8\frac{1+\tau_\mathbf{B}^2}{1-\tau_\mathbf{B}^2}\normm{\mathbf{I}_{nd}-\frac{\mathbf{v}\mathbf{1}^\intercal}{n}\otimes\mathbf{I}_{d}}_\Z^2\X^4\left(C_fL_g^2+C_gL_f^2\right)\normm{\tilde{\mathbf{A}}-\mathbf{I}_{nd}}^2,\notag\\
&c_3=8\frac{1+\tau_\mathbf{B}^2}{1-\tau_\mathbf{B}^2}\normm{\mathbf{I}_{nd}-\frac{\mathbf{v}\mathbf{1}^\intercal}{n}\otimes\mathbf{I}_{d}}_\Z^2\X^2\left(C_fL_g^2+C_gL_f^2\right) \frac{c_b^2nC_gC_f}{(1-\tau_\Z)^2},\quad c_4=\frac{1-\tau_\mathbf{A}^2}{4c_2},\notag\\
\label{consensus-para}&\tau_\mathbf{A}\define\normm{\tilde{\mathbf{A}}-\frac{\mathbf{1}\mathbf{u}^\intercal}{n}\otimes \mathbf{I}_d}_\mathbf{A},\quad\tau_\mathbf{B}\define\normm{\tilde{\mathbf{B}}-\frac{\mathbf{v}\mathbf{1}^\intercal}{n}\otimes \mathbf{I}_d}_\mathbf{B},
\end{align}
$\mathbf{u}$ and $\mathbf{v}$ are the left eigenvector of $\A$ and the right eigenvector of $\Z$ respectively, vector norms $\|\cdot\|_\A$, $\|\cdot\|_\Z$ and matrix norms $\normm{\cdot}_\A$, $\normm{\cdot}_\Z$ are introduced in Lemma \ref{lem:norm}.
\end{lem}
\begin{proof}
We first provide the upper bound of consensus error $\mathbf{x}_{k+1}-\mathbf{1}\otimes\bar{x}_{k+1}$ in the mean square sense. Note that for  any random vectors $\theta$, $\theta^{'}$ and positive scalar $\tau$, 	
\begin{equation}\label{general-form}
\begin{aligned}
\mathbb{E}\left[\left\|\theta+\theta^{'}\right\|_*^2\right]
\le (1+\tau)\mathbb{E}\left[\|\theta\|_*^2\right]+\left(1+\frac{1}{\tau}\right)\mathbb{E}\left[\left\|\theta^{'}\right\|_*^2\right],
\end{aligned}
\end{equation}
where the norm $\|\cdot\|_*$ may be $\|\cdot\|_\A$ or $\|\cdot\|_\Z$.	
Choosing
\begin{align*}
\theta=\left(\tilde{\mathbf{A}}-\frac{\mathbf{1}\mathbf{u}^\intercal}{n}\otimes\mathbf{I}_{d}\right)\left(\mathbf{x}_k-\mathbf{1}\otimes\bar{x}_{k}\right),\quad
\theta^{'}=-\alpha_k\left(\tilde{\mathbf{A}}-\frac{\mathbf{1}\mathbf{u}^\intercal}{n}\otimes\mathbf{I}_{d}\right)\mathbf{y}_k,
\end{align*}
we have $\mathbf{x}_{k+1}-\mathbf{1}\otimes\bar{x}_{k+1}=\theta+\theta^{'}$
and
\begin{align}
&\mathbb{E}\left[\|\mathbf{x}_{k+1}-\mathbf{1}\otimes\bar{x}_{k+1}\|_\A^2\right]\notag\\
&\le  (1+\tau)\mathbb{E}\left[\left\|\left(\tilde{\mathbf{A}}-\frac{\mathbf{1}\mathbf{u}^\intercal}{n}\otimes\mathbf{I}_{d}\right)\left(\mathbf{x}_k-\mathbf{1}\otimes\bar{x}_{k}\right)\right\|_\A^2\right]\notag
+\left(1+\frac{1}{\tau}\right)\mathbb{E}\left[\left\|\alpha_{k}\left(\tilde{\mathbf{A}}-\frac{\mathbf{1}\mathbf{u}^\intercal}{n}\otimes\mathbf{I}_{d}\right)\mathbf{y}_k\right\|_\A^2\right]\notag\\
&\le\frac{1+\tau_\mathbf{A}^2}{2}\mathbb{E}\left[\left\|\mathbf{x}_k-\mathbf{1}\otimes\bar{x}_{k}\right\|_\A^2\right] +\alpha_{k}^2\frac{1+\tau_\mathbf{A}^2}{1-\tau_\mathbf{A}^2}\normm{\mathbf{A}-\frac{\mathbf{1}\mathbf{u}^\intercal}{n}}_\A^2\X^2\mathbb{E}\left[\left\|\mathbf{y}_k\right\|^2\right],\label{x-bar}
\end{align}
where $\tau_\A$ is defined in (\ref{consensus-para}), $\tau=(1-\tau_\A^2)/(2\tau_\A^2)$ and the last inequality follows from the fact (\ref{norm-bound-1}). By the definition of $\mathbf{y}_k$ in (\ref{alg:new form}),
{\small\begin{equation*}
\begin{aligned}
\mathbb{E}\left[\left\|\mathbf{y}_k\right\|^2\right]
=\mathbb{E}\left[\left\|\sum_{t=1}^{k-1}\tilde{\mathbf{B}}^{k-1-t}(\tilde{\mathbf{B}}-\mathbf{I}_{nd})\mathbf{H}_t+\mathbf{H}_{k}\right\|^2\right]
\le \sum_{t_1=1}^{k}\sum_{t_2=1}^{k}\normm{\tilde{\mathbf{B}}(k,t_1)}\normm{\tilde{\mathbf{B}}(k,t_2)}\mathbb{E}\left[\|\mathbf{H}_{t_1}\|\|\mathbf{H}_{t_2}\|\right],
\end{aligned}
\end{equation*}}		
where
\begin{equation}\label{mat-B}
\begin{aligned}
&\tilde{\mathbf{B}}(k,t)\define \tilde{\mathbf{B}}^{k-1-t}(\tilde{\mathbf{B}}-\mathbf{I}_{nd})~(t\le k-1),\quad\tilde{\mathbf{B}}(k,k)\define\mathbf{I}_{nd}.
\end{aligned}
\end{equation}
Obviously, $\normm{\tilde{\mathbf{B}}(k,k)}=1$, $\normm{\tilde{\mathbf{B}}(k,k-1)}\le \X\normm{\tilde{\mathbf{B}}-\mathbf{I}_{nd}}_\Z$ and for $t<k-1$,
\begin{equation*}
\begin{aligned}
\normm{\tilde{\mathbf{B}}(k,t)}\le \X\normm{\tilde{\mathbf{B}}^{k-1-t}\left(\tilde{\mathbf{B}}-\mathbf{I}_{nd}\right)}_\Z&=\X\normm{\left(\tilde{\mathbf{B}}-\frac{\mathbf{v}\mathbf{1}^\intercal}{n}\otimes \mathbf{I}_d\right)\tilde{\mathbf{B}}^{k-2-t}\left(\tilde{\mathbf{B}}-\mathbf{I}_{n}\right)}_\Z\\
&\le\X\tau_\Z \normm{\tilde{\mathbf{B}}^{k-2-t}\left(\tilde{\mathbf{B}}-\mathbf{I}_{nd}\right)}_\Z\\
&\le\cdots\le \X\tau_\Z^{k-1-t} \normm{\tilde{\mathbf{B}}-\mathbf{I}_{nd}}_\Z.
\end{aligned}
\end{equation*}
Denoting $c_b=\max\left\{\X,\frac{\normm{\mathbf{B}-\mathbf{I}_{n}}_\Z}{\tau_\Z}\X\right\}$, we have
\begin{equation}\label{B-bound}
\normm{\tilde{\mathbf{B}}(k,t)}\le c_b\tau_\Z^{k-t}
\end{equation}
 and
\begin{align}
\mathbb{E}\left[\left\|\mathbf{y}_k\right\|^2\right]
&\le c_b^2\sum_{t_1=1}^{k}\sum_{t_2=1}^{k}\tau_\Z^{2k-t_1-t_2}\mathbb{E}\left[\|\mathbf{H}_{t_1}\|\|\mathbf{H}_{t_2}\|\right]\notag\\
&\le c_b^2\sum_{t_1=1}^{k}\sum_{t_2=1}^{k}\tau_\Z^{2k-t_1-t_2}\frac{\mathbb{E}\left[\|\mathbf{H}_{t_1}\|^2\right]+\mathbb{E}\left[\|\mathbf{H}_{t_2}\|^2\right]}{2}\notag\\
&\le c_b^2\sum_{t_1=1}^{k}\sum_{t_2=1}^{k}\tau_\Z^{2k-t_1-t_2} nC_gC_f\le \frac{c_b^2nC_gC_f}{(1-\tau_\Z)^2},\label{y-bound}
\end{align}	
where  the third inequality follows from Assumption \ref{ass-objective} (c). Substitute (\ref{y-bound}) into (\ref{x-bar}),
\begin{equation}\label{ine-1}
\mathbb{E}\left[\|\mathbf{x}_{k+1}-\mathbf{1}\otimes\bar{x}_{k+1}\|_\A^2\right]
\le\frac{1+\tau_\mathbf{A}^2}{2}\mathbb{E}\left[\left\|\mathbf{x}_k-\mathbf{1}\otimes\bar{x}_{k}\right\|_\A^2\right]+c_1\alpha_{k}^2,
\end{equation}
where $c_1=\frac{1+\tau_\mathbf{A}^2}{1-\tau_\mathbf{A}^2}\normm{\mathbf{A}-\frac{\mathbf{1}\mathbf{u}^\intercal}{n}}_\A^2\X^2\frac{c_b^2nC_gC_f}{(1-\tau_\Z)^2}$.

Next, we estimate the upper bound of consensus error $\|\mathbf{y}_k^{'}-\mathbf{v}\otimes\bar{y}_{k}^{'}\|^2$ in the mean sense. Set
\begin{align*}
\theta=\left(\tilde{\mathbf{B}}-\frac{\mathbf{v}\mathbf{1}^\intercal}{n}\otimes\mathbf{I}_{d}\right)\left(\mathbf{y}_k^{'}-\mathbf{v}\otimes \bar{y}_{k}^{'}\right),\quad
\theta^{'}=\left(\mathbf{I}_{nd}-\frac{\mathbf{v}\mathbf{1}^\intercal}{n}\otimes\mathbf{I}_{d}\right)\left(\mathbf{J}_{k+1}-\mathbf{J}_k\right)
\end{align*}
in (\ref{general-form}).  By the definitions of $\mathbf{y}_{k+1}^{'}$ and $\bar{y}_{k+1}^{'}$, we have $\mathbf{y}_{k+1}^{'}-\mathbf{1}\otimes\bar{y}_{k+1}^{'}=\theta+\theta^{'}$ and
{\small\begin{align}
&\mathbb{E}\left[\|\mathbf{y}_{k+1}^{'}-\mathbf{v}\otimes\bar{y}_{k+1}^{'}\|_\Z^2\right]\notag\\
&\le(1+\tau)\mathbb{E}\left[\left\|\left(\tilde{\mathbf{B}}-\frac{\mathbf{v}\mathbf{1}^\intercal}{n}\otimes\mathbf{I}_{d}\right)\left(\mathbf{y}_k^{'}-\mathbf{v}\otimes\bar{y}_{k}^{'}\right)\right\|_\Z^2\right]+\left(1+\frac{1}{\tau}\right)\mathbb{E}\left[\left\|\left(\mathbf{I}_{nd}-\frac{\mathbf{v}\mathbf{1}^\intercal}{n}\otimes\mathbf{I}_{d}\right)\left(\mathbf{J}_{k+1}-\mathbf{J}_k\right)\right\|_\Z^2\right]\notag\\
\label{consensus-3}&\le
\frac{1+\tau_\mathbf{B}^2}{2}\mathbb{E}\left[\left\|\mathbf{y}_k^{'}-\mathbf{v}\otimes \bar{y}_{k}^{'}\right\|_\Z^2\right]+2\frac{1+\tau_\mathbf{B}^2}{1-\tau_\mathbf{B}^2}\normm{\mathbf{I}_{nd}-\frac{\mathbf{v}\mathbf{1}^\intercal}{n}\otimes\mathbf{I}_{d}}_\Z^2\X^2\mathbb{E}\left[\left\|\mathbf{J}_{k+1}-\mathbf{J}_k\right\|^2\right],
\end{align}}
where $\tau_\Z$ is defined in (\ref{consensus-para}), the second inequality follows from the setting $\tau=(1-\tau_\Z^2)/(2\tau_\Z^2)$ and (\ref{norm-bound-1}). For the term $\mathbb{E}\left[\left\|\mathbf{J}_{k+1}-\mathbf{J}_k\right\|^2\right]$,
\begin{align}
\mathbb{E}\left[\left\|\mathbf{J}_{k+1}-\mathbf{J}_k\right\|^2\right]
&=\mathbb{E}\left[\left\|\left(\nabla \mathbf{g}_{k+1}-\nabla \mathbf{g}_k\right)\nabla \mathbf{f}_{k+1}+\nabla \mathbf{g}_k\left(\nabla \mathbf{f}_{k+1}-\nabla \mathbf{f}_k\right)\right\|^2\right]\notag\\
&\le 2\left(C_fL_g^2+C_gL_f^2\right)\mathbb{E}\left[\left\|\mathbf{x}_{k+1}-\mathbf{x}_k\right\|^2\right]\notag\\
&=2\left(C_fL_g^2+C_gL_f^2\right) \mathbb{E}\left[\left\|\left(\tilde{\mathbf{A}}-\mathbf{I}_{nd}\right)\left(\mathbf{x}_k-\mathbf{1}\otimes \bar{x}_k\right)-\alpha_k\tilde{\mathbf{A}}\mathbf{y}_k\right\|^2\right]\notag\\
&\le 4\left(C_fL_g^2+C_gL_f^2\right)\normm{\tilde{\mathbf{A}}-\mathbf{I}_{nd}}^2 \X^2\mathbb{E}\left[ \left\|\mathbf{x}_k-\mathbf{1}\otimes \bar{x}_k\right\|_\A^2\right]\notag\\
\label{F-error}&\quad+4\left(C_fL_g^2+C_gL_f^2\right) ^2\alpha_k^2\mathbb{E}\left[\left\|\mathbf{y}_k\right\|^2\right],
\end{align}	
where $\nabla \mathbf{g}_k=\left[\nabla g_1(x_{1,k})^\intercal,\cdots, \nabla g_n(x_{n,k})^\intercal\right]^\intercal$ and $\nabla \mathbf{f}_k=\left[\nabla f_1(x_{1,k})^\intercal,\cdots, \nabla f_n(x_{n,k})^\intercal\right]^\intercal$, the first inequality follows from Assumption \ref{ass-objective} (a) and (c), the second equality follows from the fact $\left(\tilde{\mathbf{A}}-\mathbf{I}_{nd}\right)(\mathbf{1}\otimes \bar{x}_k)=\mathbf{0}$ as $\mathbf{A}$ is a row stochastic matrix. Substitute (\ref{y-bound}) and (\ref{F-error}) into (\ref{consensus-3}),
\begin{align}
&\mathbb{E}\left[\|\mathbf{y}_{k+1}^{'}-\mathbf{v}\otimes\bar{y}_{k+1}^{'}\|_\Z^2\right]\notag\\
&\le
\frac{1+\tau_\mathbf{B}^2}{2}\mathbb{E}\left[\left\|\mathbf{y}_k^{'}-\mathbf{v}\otimes \bar{y}_{k}^{'}\right\|_\Z^2\right]+c_2\mathbb{E}\left[\left\|\mathbf{x}_k-\mathbf{1}\otimes \bar{x}_{k}\right\|_\A^2\right]+c_3\alpha_k^2,\label{ine-2}
\end{align}
where the constants
\begin{equation*}\label{parameter-2}
\begin{aligned}
&c_2=8\frac{1+\tau_\mathbf{B}^2}{1-\tau_\mathbf{B}^2}\normm{\mathbf{I}_{nd}-\frac{\mathbf{v}\mathbf{1}^\intercal}{n}\otimes\mathbf{I}_{d}}_\Z^2\X^4\left(C_fL_g^2+C_gL_f^2\right)\normm{\tilde{\mathbf{A}}-\mathbf{I}_{nd}}^2,\\
&c_3=8\frac{1+\tau_\mathbf{B}^2}{1-\tau_\mathbf{B}^2}\normm{\mathbf{I}_{nd}-\frac{\mathbf{v}\mathbf{1}^\intercal}{n}\otimes\mathbf{I}_{d}}_\Z^2\X^2\left(C_fL_g^2+C_gL_f^2\right)\normm{\mathbf{A}}^2 \frac{c_b^2nC_gC_f}{(1-\tau_\Z)^2}.
\end{aligned}
\end{equation*}
	
Lastly, we show (\ref{ine-0}) through combining (\ref{ine-1}) with (\ref{ine-2}). Multiplying $c_4=\frac{1-\tau_\mathbf{A}^2}{4c_2}$ on both sides of inequality (\ref{ine-2}),
{\small\begin{equation*}
	\begin{aligned}
	&c_4\mathbb{E}\left[\|\mathbf{y}_{k+1}^{'}-\mathbf{v}\otimes\bar{y}_{k+1}^{'}\|_\Z^2\right]\\
	&\le \frac{1+\tau_\mathbf{B}^2}{2}c_4\mathbb{E}\left[\left\|\mathbf{y}_k^{'}-\mathbf{v}\otimes \bar{y}_{k}^{'}\right\|_\Z^2\right]+\frac{1-\tau_\mathbf{A}^2}{4}\mathbb{E}\left[\left\|\mathbf{x}_k-\mathbf{1}\otimes \bar{x}_k\right\|_\A^2\right]+c_3c_4\alpha_k^2.
	\end{aligned}
	\end{equation*}}
Substituting above inequality into (\ref{ine-1}), we have
\begin{equation*}\label{ine-3}
\begin{aligned}
&\mathbb{E}\left[\|\mathbf{x}_{k+1}-\mathbf{1}\otimes\bar{x}_{k+1}\|_\A^2\right]+c_4\mathbb{E}\left[\|\mathbf{y}_{k+1}^{'}-\mathbf{v}\otimes\bar{y}_{k+1}^{'}\|_\Z^2\right]\\
&\le \frac{3+\tau_\mathbf{A}^2}{4}\mathbb{E}\left[\left\|\mathbf{x}_k-\mathbf{1}\otimes \bar{x}_k\right\|_\A^2\right]+\frac{1+\tau_\mathbf{B}^2}{2}c_4\mathbb{E}\left[\left\|\mathbf{y}_k^{'}-\mathbf{v}\otimes \bar{y}_{k}^{'}\right\|_\Z^2\right]+(c_1+c_3c_4)\alpha_k^2\\
&\le \rho \left(\mathbb{E}\left[\left\|\mathbf{x}_k-\mathbf{1}\otimes \bar{x}_k\right\|_\A^2\right]+c_4\mathbb{E}\left[\left\|\mathbf{y}_k^{'}-\mathbf{v}\otimes \bar{y}_{k}^{'}\right\|_\Z^2\right]\right)+(c_1+c_3c_4)\alpha_k^2\\
&\cdots\\
&\le\rho^{k} \left(\mathbb{E}\left[\left\|\mathbf{x}_1-\mathbf{1}\otimes \bar{x}_1\right\|_\A^2\right]+c_4\mathbb{E}\left[\left\|\mathbf{y}_{1}^{'}-\mathbf{v}\otimes \bar{y}_{1}^{'}\right\|_\Z^2\right]\right)+(c_1+c_3c_4)\sum_{t=1}^k\rho^{k-t}\alpha_t^2,
\end{aligned}
\end{equation*}
where $\rho=\max\left\{\frac{1+\tau_\mathbf{B}^2}{2},~\frac{3+\tau_\mathbf{A}^2}{4}\right\}$.  The proof is complete.
\end{proof}

	Lemma \ref{lem:rate} indicates that the consensus errors can be explicitly
	decomposed into “bias” and “variance” terms. The bias term characterizes how fast initial
	conditions are forgotten and is related to condition numbers $\tau_\A$ and $\tau_\Z$ of network topology.
	The variance term characterizes the effect of new stochastic gradient, which is independent
	of the starting point and increases with the gradient upper bounds $C_f,C_g$ and Lipschitz parameters $L_f,L_g$.
	
	The following lemma is a technical result.
\begin{lem}\label{lem:weighted-seq}
	Suppose that stepsize $\alpha_k$ is nonincreasing and $\lim_{k\rightarrow\infty}\frac{\alpha_k}{\alpha_{k+1}}=1$. Then there exists a constant $c$ such that
	\begin{equation*}
	\sum_{t=1}^k\rho^{k-t}\alpha_t\le c\alpha_k,
	\end{equation*}
	where scalar $\rho=\max\left\{\frac{1+\tau_\mathbf{B}^2}{2},~\frac{3+\tau_\mathbf{A}^2}{4}\right\}$.
\end{lem}
\begin{proof}
	Let $\beta_k=\sum_{t=1}^{k}\rho^{k-t}\alpha_t$, then $\beta_k=\rho\sum_{t=1}^{k-1}\rho^{k-1-t}\alpha_t+\alpha_{k}=\rho \beta_{k-1}+\alpha_{k}$.
	Denoting $b_k=\beta_k/\alpha_{k}$, then $b_k=\rho \frac{\alpha_{k-1}}{\alpha_{k}}b_{k-1}+1$.
	Noting that $\lim_{k\rightarrow\infty}\frac{\alpha_{k-1}}{\alpha_{k}}=1$ and $\rho<1$, there exists an integer $k_0>0$ such that $\frac{\alpha_{k-1}}{\alpha_{k}}\le \frac{2}{\rho+1}$ for $k>k_0$. Taking $c=\max\left\{\sup_{1\le k\le k_0}b_k,~\frac{\rho+1}{1-\rho}\right\}$,
	we have  $b_k\le c$ for $k\le k_0$. Suppose
	that the claim holds for $k-1$ ($k-1\ge k_0$), that is $b_{k-1}\le c$, then
	\begin{equation*}
	b_k=\rho \frac{\alpha_{k-1}}{\alpha_{k}}b_{k-1}+1\le \frac{2\rho}{\rho+1} c+1\le \frac{2\rho}{\rho+1} c+\frac{1-\rho}{\rho+1}c=c.
	\end{equation*}
	The proof is complete.	
\end{proof}

By the fact  $\lim_{k\rightarrow\infty}\frac{\alpha_k^2}{\alpha_{k+1}^2}=1$ and Lemma \ref{lem:weighted-seq}, (\ref{ine-0}) can be rewritten as
\begin{align*}
&\mathbb{E}\left[\|\mathbf{x}_{k+1}-\mathbf{1}\otimes\bar{x}_{k+1}\|_\A^2\right]+c_4\mathbb{E}\left[\|\mathbf{y}_{k+1}^{'}-\mathbf{v}\otimes\bar{y}_{k+1}^{'}\|_\Z^2\right]\\
&\le\rho^{k} \left(c_4\mathbb{E}\left[\left\|\mathbf{y}_{1}^{'}-\mathbf{v}\otimes \bar{y}_{1}^{'}\right\|_\Z^2\right]+\mathbb{E}\left[\left\|\mathbf{x}_1-\mathbf{1}\otimes \bar{x}_1\right\|_\A^2\right]\right)+(c_1+c_3c_4)c\alpha_k^2.
\end{align*}
Assuming $\rho^k=o\left(\alpha_k^2\right)$\footnote{In this paper, we use  constant stepsize and  sublinear diminishing stepsize.  Then  $\rho^k=o\left(\alpha_k^2\right)$ holds.}, the consensus errors have a rough upper bounds
\begin{equation}\label{x-y-consensus}
\mathbb{E}\left[\|\mathbf{x}_{k}-\mathbf{1}\otimes\bar{x}_{k}\|_\A^2\right]\le U_1\alpha_k^2,\quad
\mathbb{E}\left[\|\mathbf{y}_k^{'}-\mathbf{v}\otimes\bar{y}_{k}^{'}\|^2\right]\le U_1\alpha_k^2,
\end{equation}
where the constant $U_1$ depends on parameters $\tau_\A,\tau_\Z,C_f,C_g,L_f$ and $L_g$.

The next lemma quantifies the error of estimating $g_i(x_{i,k})$ by $z_{i,k}$.
\begin{lem}\label{lem:g-tra}
 Suppose that stepsizes $\alpha_k$ and $\beta_k$ are nonincreasing and $\lim_{k\rightarrow\infty}\frac{\alpha_k}{\alpha_{k+1}}=1$, $\beta_k\le 1$. Then under Assumptions \ref{ass-objective}-\ref{ass:matrix},
{\begin{align}
\mathbb{E}\left[\|\mathbf{z}_{k+1}- \mathbf{g}_{k+1}\|^2\right] &\le(1-\beta_k)^2\mathbb{E}\left[\|\mathbf{z}_k-\mathbf{g}_k\|^2\right]+\left(12C_g\X^2\normm{\tilde{\mathbf{A}}-\mathbf{I}_{nd}}^2U_1\right.\notag\\
\label{z-convergence}&\quad\left.+\frac{12c_b^2nC_g^2C_f\normm{\mathbf{A}}^2}{(1-\tau_\Z)^2}\right)\alpha_k^2+3V_g\beta_k^2,
\end{align}	}
where $ \mathbf{g}_k=\left[ g_1(x_{1,k})^\intercal,\cdots,  g_n(x_{n,k})^\intercal\right]^\intercal$ and $\X$ is defined in (\ref{norm-bound-1}).
\end{lem}
\begin{proof}
By the definitions of $\mathbf{z}_{k+1}$ and $\mathbf{g}_{k+1}$,
\begin{equation}\label{z-recur}
\mathbf{z}_{k+1}- \mathbf{g}_{k+1} =(1-\beta_k)\left(\mathbf{z}_k-\mathbf{g}_k\right)+(\mathbf{G}_{k+1}^{(1)}-\mathbf{g}_{k+1})+(1-\beta_k)(\mathbf{g}_k-\mathbf{G}_{k+1}^{(2)}).
\end{equation}
Then
\begin{align}
&\mathbb{E}\left[\|\mathbf{z}_{k+1}- \mathbf{g}_{k+1}\|^2\right]\notag\\
&= (1-\beta_k)^2\mathbb{E}\left[\|\mathbf{z}_k-\mathbf{g}_k\|^2\right]+\mathbb{E}\left[\|(\mathbf{G}_{k+1}^{(1)}-\mathbf{g}_{k+1})+(1-\beta_k)(\mathbf{g}_k-\mathbf{G}_{k+1}^{(2)})\|^2\right]\notag\\
&\quad + 2\mathbb{E}\left[\left\langle(1-\beta_k)(\mathbf{z}_k-\mathbf{g}_k),(\mathbf{G}_{k+1}^{(1)}-\mathbf{g}_{k+1})+(1-\beta_k)(\mathbf{g}_k-\mathbf{G}_{k+1}^{(2)})\right\rangle\right]\notag\\
\label{z-bound}&= (1-\beta_k)^2\mathbb{E}\left[\|\mathbf{z}_k-\mathbf{g}_k\|^2\right]+\mathbb{E}\left[\|(\mathbf{G}_{k+1}^{(1)}-\mathbf{g}_{k+1})+(1-\beta_k)(\mathbf{g}_k-\mathbf{G}_{k+1}^{(2)})\|^2\right],
\end{align}
where the second equality follows from the fact
{\small\begin{equation*}
\begin{aligned}
\mathbb{E}\left[(\mathbf{G}_{k+1}^{(1)}-\mathbf{g}_{k+1})+(1-\beta_k)(\mathbf{g}_k-\mathbf{G}_{k+1}^{(2)})\right]=\mathbb{E}\left[\mathbb{E}\left[(\mathbf{G}_{k+1}^{(1)}-\mathbf{g}_{k+1})+(1-\beta_k)(\mathbf{g}_k-\mathbf{G}_{k+1}^{(2)})\bigg|\mathcal{F}_k^{'}\right]\right]=\mathbf{0}
\end{aligned}
\end{equation*}}
with
\begin{equation}\label{sigma-alge}
\begin{aligned}
&\mathcal{F}_1^{'}=\sigma\left(x_{i,1}, z_{i,1}, \phi_{i,1},\zeta_{i,1}:i\in\mathcal{V}\right),\\
&\mathcal{F}_k^{'}=\sigma\left(\{x_{i,1},z_{i,1}, \phi_{i,t},\zeta_{i,t}:i\in\mathcal{V}, 1\le t\le k\}\cup\{\phi_{i,t}^{'}:i\in\mathcal{V}, 2\le t\le k\}\right)(k\ge2).
\end{aligned}
\end{equation}
For the second term on the right hand side of (\ref{z-bound}),
\begin{align*}
&\mathbb{E}\left[\|(\mathbf{G}_{k+1}^{(1)}-\mathbf{g}_{k+1})+(1-\beta_k)(\mathbf{g}_k-\mathbf{G}_{k+1}^{(2)})\|^2\right]\notag\\
&=\mathbb{E}\left[\|(1-\beta_k)(\mathbf{G}_{k+1}^{(1)}-\mathbf{G}_{k+1}^{(2)})+\beta_k( \mathbf{G}_{k+1}^{(1)}- \mathbf{g}_{k+1})+(1-\beta_k)(\mathbf{g}_k-\mathbf{g}_{k+1})	\|^2\right]\notag\\
&\le 3(1-\beta_k)^2\mathbb{E}\left[\|\mathbf{G}_{k+1}^{(1)}-\mathbf{G}_{k+1}^{(2)}\|^2\right]+3\beta_k^2\mathbb{E}\left[\|\mathbf{G}_{k+1}^{(1)}- \mathbf{g}_{k+1}\|^2\right]+3(1-\beta_k)^2\mathbb{E}\left[\|\mathbf{g}_k-\mathbf{g}_{k+1}	 \|^2\right]\notag\\
\label{z-bound-1}&\le 6(1-\beta_k)^2C_g\mathbb{E}\left[\|\mathbf{x}_{k+1}-\mathbf{x}_k\|^2\right]+3\beta_k^2V_g,
\end{align*}
where the second inequality follows from  the conditions (c) and (d) in  Assumption \ref{ass-objective}. Substitute above inequality into (\ref{z-bound}),
\begin{align}
&\mathbb{E}\left[\|\mathbf{z}_{k+1}- \mathbf{g}_{k+1}\|^2\right]\notag\\
&\le (1-\beta_k)^2\mathbb{E}\left[\|\mathbf{z}_k-\mathbf{g}_k\|^2\right]+6(1-\beta_k)^2C_g\mathbb{E}\left[\|\mathbf{x}_{k+1}-\mathbf{x}_k\|^2\right]+3\beta_k^2V_g\notag\\
&=(1-\beta_k)^2\mathbb{E}\left[\|\mathbf{z}_k-\mathbf{g}_k\|^2\right]+6(1-\beta_k)^2C_g\mathbb{E}\left[\left\|\left(\tilde{\mathbf{A}}-\mathbf{I}_{nd}\right)\left(\mathbf{x}_k-\mathbf{1}\otimes \bar{x}_k\right)-\alpha_k\tilde{\mathbf{A}}\mathbf{y}_k\right\|^2\right]+3\beta_k^2V_g\notag\\
&\le(1-\beta_k)^2\mathbb{E}\left[\|\mathbf{z}_k-\mathbf{g}_k\|^2\right]+12(1-\beta_k)^2C_g\X^2\normm{\tilde{\mathbf{A}}-\mathbf{I}_{nd}}^2\mathbb{E}\left[ \left\|\mathbf{x}_k-\mathbf{1}\otimes \bar{x}_k\right\|_\A^2\right]\notag\\
&\quad+12(1-\beta_k)^2C_g\alpha_k^2\normm{\mathbf{A}}^2\mathbb{E}\left[\left\|\mathbf{y}_k\right\|^2\right]+3\beta_k^2V_g\notag\\
&\le(1-\beta_k)^2\mathbb{E}\left[\|\mathbf{z}_k-\mathbf{g}_k\|^2\right]+\left(12C_g\X^2\normm{\tilde{\mathbf{A}}-\mathbf{I}_{nd}}^2U_1+\frac{12c_b^2nC_g^2C_f\normm{\mathbf{A}}^2}{(1-\tau_\Z)^2}\right)\alpha_k^2+3V_g\beta_k^2\notag,
\end{align}
where $\X$ is defined in (\ref{norm-bound-1}), the equality follows from the fact $\left(\tilde{\mathbf{A}}-\mathbf{I}_{nd}\right)(\mathbf{1}\otimes \bar{x}_k)=\mathbf{0}$ by the row stochasticity of $\mathbf{A}$, the last inequality follows from (\ref{y-bound}), (\ref{x-y-consensus}) and the definition of $\beta_k$. 
The proof is complete.
\end{proof}
The following lemma studies the boundness  of stochastic noise accumulated in gradient tracking process.
\begin{lem}\label{lem:nonconv}
	Define
	\begin{equation}\label{noise}
	\xi_k\define y_k-y_k^{'}.
	\end{equation}
	Under the conditions of Lemma \ref{lem:rate}, 
	\begin{itemize}
		\item [(i)]$\mathbb{E}\left[\left\|\xi_k\right\|^2\right]\le \frac{c_b^24n C_fC_g}{(1-\tau_\Z)^2}$, where $c_b=\max\left\{\X,\frac{\normm{\mathbf{B}-\mathbf{I}_{n}}_\Z}{\tau_\Z}\X\right\}$;
		\item [(ii)] there exists constant $U_3>0$ such that
		\begin{equation*}
		\left|\mathbb{E}\left[\left\langle \nabla h(\bar{x}_k),\left(\frac{\mathbf{u}^\intercal}{n}\otimes\mathbf{I}_{d}\right)\xi_k\right\rangle\right]\right|\le U_3\alpha_k.
		\end{equation*}
	\end{itemize}
\end{lem}
\begin{proof}We first show part (i). By the definition of $\xi_k$,
	\begin{equation}\label{e-repre}
	\xi_k=\sum_{t=1}^{k-1}\tilde{\mathbf{B}}^{k-1-t}(\tilde{\mathbf{B}}-\mathbf{I}_{nd})\epsilon_t+\epsilon_{k}=\sum_{t=1}^{k}\tilde{\mathbf{B}}(k,t)\epsilon_t,
	\end{equation}
	where $\epsilon_t\define\mathbf{H}_t-\mathbf{J}_t$, $\mathbf{H}_t$ and $\mathbf{J}_t$ present in (\ref{alg:new form}) and Lemma \ref{lem:rate} respectively, $\tilde{\mathbf{B}}(k,t)$ is defined in (\ref{mat-B}).
	Then we have
	\begin{align}
	\mathbb{E}[\|\xi_k\|^2]
	&\le\sum_{t_1=1}^{k}\sum_{t_2=1}^{k}\normm{\tilde{\mathbf{B}}(k,t_1)}\normm{\tilde{\mathbf{B}}(k,t_2)}\mathbb{E}\left[\|\epsilon_{t_1}\|\|\epsilon_{t_2}\|\right]\notag\\
	&\le c_b^2\sum_{t_1=1}^{k}\sum_{t_2=1}^{k}\tau_\Z^{2k-t_1-t_2}\mathbb{E}\left[\|\epsilon_{t_1}\|\|\epsilon_{t_2}\|\right]\notag\\
	\label{xi-bound}&\le c_b^2\sum_{t_1=1}^{k}\sum_{t_2=1}^{k}\tau_\Z^{2k-t_1-t_2}\frac{\mathbb{E}\left[\|\epsilon_{t_1}\|^2+\|\epsilon_{t_2}\|^2\right]}{2},
	\end{align}
	where $c_b=\max\left\{\X,\frac{\normm{\mathbf{B}-\mathbf{I}_{n}}_\Z}{\tau_\Z}\X\right\}$, the second inequality follows from (\ref{B-bound}). By the definition of $\epsilon_k$,
	\begin{align}
	\mathbb{E}\left[\|\epsilon_{k}\|^2\right]
	&=\sum_{j=1}^n\mathbb{E}\left[\|\nabla G_j(x_{j,k};\phi_{j,k})\nabla F_j(z_{j,k};\zeta_{j,k})-\nabla g_j(x_{j,k})\nabla f_j(z_{j,k})\|^2\right]\notag\\
	&\le 2\sum_{j=1}^n\left(\mathbb{E}\left[\|\nabla G_j(x_{j,k};\phi_{j,k})\|^2\|\nabla F_j(z_{j,k};\zeta_{j,k})\|^2\right]+C_fC_g\right)\notag\\
	&= 2\sum_{j=1}^n\left(\mathbb{E}\left[\mathbb{E}\left[\|\nabla G_j(x_{j,k};\phi_{j,k})\|^2\|\nabla F_j(z_{j,k};\zeta_{j,k})\|^2\big|\mathcal{F}_k,\zeta_{j,k}\right]\right]+C_fC_g\right)\notag\\
	\label{e-bound}&\le 2\sum_{j=1}^n\left(C_g\mathbb{E}\left[\|\nabla F_j(z_{j,k};\zeta_{j,k})\|^2\right]+C_fC_g\right)\le 4n C_fC_g,
	\end{align}
	where 
    \begin{equation}\label{s-algebra}
    \begin{aligned}
    &\mathcal{F}_1=\sigma\{x_{i,1}, z_{i,1}:i\in\mathcal{V}\},\\
    &\mathcal{F}_k=\sigma\left(\{x_{i,1},z_{i,1}, \phi_{i,t},\zeta_{i,t}:i\in\mathcal{V}, 1\le t\le k-1\}\cup\{\phi_{i,t}^{'}:i\in\mathcal{V}, 2\le t\le k\}\right)(k\ge2).
    \end{aligned}
	\end{equation}
 Substitute (\ref{e-bound}) into (\ref{xi-bound}), $\mathbb{E}[\|\xi_k\|^2]
 \le c_b^24n C_fC_g\sum_{t_1=1}^{k}\sum_{t_2=1}^{k}\tau_\Z^{2k-t_1-t_2}\le \frac{c_b^24n C_fC_g}{(1-\tau_\Z)^2}$.
	Part (i) is obtained.
	
	By  (\ref{e-repre}),
	{\small\begin{align*}
	&\left|\mathbb{E}\left[\left\langle \nabla h(\bar{x}_k),\left(\frac{\mathbf{u}^\intercal}{n}\otimes\mathbf{I}_{d}\right)\xi_k\right\rangle\right]\right|\\
	&=\left|\sum_{t=1}^{k}\mathbb{E}\left[\left\langle \nabla h(\bar{x}_k),\left(\frac{\mathbf{u}^\intercal}{n}\otimes\mathbf{I}_{d}\right)\tilde{\mathbf{B}}(k,t)\epsilon_t\right\rangle\right]\right|\\
	&=\left|\sum_{t=1}^{k-1}\mathbb{E}\left[\mathbb{E}\left[\left\langle \sum_{l=t+1}^{k}\left(\nabla h(\bar{x}_l)-\nabla h(\bar{x}_{l-1})\right)+\nabla h(\bar{x}_t),\left(\frac{\mathbf{u}^\intercal}{n}\otimes\mathbf{I}_{d}\right)\tilde{\mathbf{B}}(k,t)\epsilon_t\right\rangle\bigg|\mathcal{F}_t\right]\right]\right.\\
	&\quad\left.+\mathbb{E}\left[\mathbb{E}\left[\left\langle \nabla h(\bar{x}_k),\left(\frac{\mathbf{u}^\intercal}{n}\otimes\mathbf{I}_{d}\right)\tilde{\mathbf{B}}(k,t)\epsilon_k\right\rangle\bigg|\mathcal{F}_k\right]\right]\right|\\
	&=\left|\sum_{t=1}^{k-1}\mathbb{E}\left[\left\langle \sum_{l=t+1}^{k}\left(\nabla h(\bar{x}_l)-\nabla h(\bar{x}_{l-1})\right),\left(\frac{\mathbf{u}^\intercal}{n}\otimes\mathbf{I}_{d}\right)\tilde{\mathbf{B}}(k,t)\epsilon_t\right\rangle\right]\right|\\
	&\le\frac{\|\mathbf{u}\|Lc_b}{n}\sum_{t=1}^{k-1}\tau_\Z^{k-t}\sum_{l=t+1}^{k}\mathbb{E}\left[\left\| \bar{x}_l-\bar{x}_{l-1}\right\|\left\|\epsilon_t\right\|\right]\\
	&\le\frac{\|\mathbf{u}\|^2Lc_b}{n^2}\sum_{t=1}^{k-1}\tau_\Z^{k-t}\sum_{l=t+1}^{k}\alpha_l\mathbb{E}\left[\left\| \mathbf{y}_l\right\|\left\|\epsilon_t\right\|\right],
	\end{align*}}
	where the third equality holds as $\{\epsilon_t\}$ is a martingale difference sequence, the first inequality follows from (\ref{B-bound}) and the last inequality follows from the fact $\bar{x}_{k+1}=\bar{x}_{k}-\alpha_k\left(\frac{\mathbf{u}^\intercal}{n}\otimes\mathbf{I}_{d}\right)\mathbf{y}_k$. By (\ref{y-bound}) and (\ref{e-bound}),
	\begin{align}
	\mathbb{E}\left[\left\| \mathbf{y}_l\right\|\left\|\epsilon_t\right\|\right]\le\frac{\mathbb{E}\left[\left\|\mathbf{y}_l\right\|^2\right]+\mathbb{E}\left[\left\|\epsilon_k\right\|^2\right]}{2}\le  \frac{c_b^2n\ C_fC_g}{2(1-\tau_\Z)^2}+2nC_fC_g.
	\end{align}
	Let $U=\frac{\|\mathbf{u}\|^2Lc_b}{n}\left(\frac{c_b^2\ C_fC_g}{2(1-\tau_\Z)^2}+2C_fC_g\right)$,
	{\small\begin{equation*}
	\mathbb{E}\left[\left\langle \nabla h(\bar{x}_k),\left(\frac{\mathbf{u}^\intercal}{n}\otimes\mathbf{I}_{d}\right)\xi_k\right\rangle\right]\le(1-\tau_\Z)U \sum_{t=1}^{k-1}\tau_\Z^{k-t}\sum_{l=t+1}^{k}\alpha_l=(1-\tau_\Z)U\sum_{t=2}^{k}\alpha_t\tau_\Z^{k-t}\left(\sum_{l=1}^{t-1}\tau_\Z^l\right)\le U c \alpha_k,
	\end{equation*}}
	where the last inequality follows from the fact $(1-\tau_\Z)\left(\sum_{l=1}^{t-1}\tau_\Z^l\right)\le 1$ and Lemma \ref{lem:weighted-seq}. Part (ii) holds. The proof is complete.
\end{proof}

With Lemmas \ref{lem:norm}-\ref{lem:nonconv} at hand, we are ready to  present the convergence rate of AB-DSCSC.
\begin{thm}\label{thm:nonconv}
	Let $\alpha_{k}=\frac{a}{\sqrt{K}}$, $\beta_k=\frac{\alpha_{k} C_gL_f^2}{n}$ and $a<\frac{n}{ C_gL_f^2}$. Then under Assumptions \ref{ass-objective}-\ref{ass:matrix},
	{\small\begin{equation*}
	\frac{1}{K}\sum_{k=1}^K\mathbb{E}\left[\|\nabla h(x_{i,k})\|^2\right]
\le  \frac{8\left(\mathbb{E}\left[h(\bar{x}_1)\right]+\mathbb{E}\left[\|\mathbf{z}_1-\mathbf{g}_1\|^2\right]\right)/a+8a U_4}{\sqrt{K}}+\frac{2\left(\frac{4L^2U_1}{n}+\frac{4\|\mathbf{u}\|^2U_1}{n^2}+L^2U_1\right)a^2}{K},
	\end{equation*}}
	where constants $U_1$ is defined in (\ref{x-y-consensus}), $U_3$ presents in  Lemma \ref{lem:nonconv} and
	\begin{equation*}
	U_4=\frac{L\|\mathbf{u}\|^2c_b^2\ C_fC_g}{2n^2(1-\tau_\Z)^2}+U_3+12C_g\X^2\normm{\tilde{\mathbf{A}}-\mathbf{I}_{nd}}^2U_1+\frac{12c_b^2nC_g^2C_f\normm{\mathbf{A}}^2}{(1-\tau_\Z)^2}+3V_g\frac{ C_g^2L_f^4}{n^2}.
	\end{equation*}
\end{thm}
\begin{proof}
	We first estimate the upper bound of $\nabla h(\bar{x}_k)$ in expectation.
	Noting that $\nabla h(x)$ is $L\left(\define C_g^2L_f + C_fL_g\right)$-smooth \cite{Junyu2019SC},
	\begin{align*}
	h(\bar{x}_{k+1})&\le h(\bar{x}_k)+\langle \nabla h(\bar{x}_k),\bar{x}_{k+1}-\bar{x}_k\rangle+\frac{L}{2}\|\bar{x}_{k+1}-\bar{x}_k\|^2\notag\\
	&=h(\bar{x}_k)-\left\langle \nabla h(\bar{x}_k),\alpha_k\left(\frac{\mathbf{u}^\intercal}{n}\otimes\mathbf{I}_{d}\right)\left(\mathbf{y}_k^{'}+\xi_k\right)\right\rangle+\frac{L}{2}\left\|\alpha_{k}\left(\frac{\mathbf{u}^\intercal}{n}\otimes\mathbf{I}_{d}\right)\mathbf{y}_k\right\|^2\notag\\
	&=h(\bar{x}_k)-\frac{\mathbf{u}^\intercal \mathbf{v}\alpha_{k}}{n}\|\nabla h(\bar{x}_k)\|^2+\frac{L}{2}\left\|\alpha_{k}\left(\frac{\mathbf{u}^\intercal}{n}\otimes\mathbf{I}_{d}\right)\mathbf{y}_k\right\|^2\notag\\
	\label{grad-bound-0}&\quad+\left\langle \nabla h(\bar{x}_k),\alpha_{k}\left(\frac{\mathbf{u}^\intercal \mathbf{v}}{n}\nabla h(\bar{x}_k)-\left(\frac{\mathbf{u}^\intercal}{n}\otimes\mathbf{I}_{d}\right)\left(\mathbf{y}_k^{'}+\xi_k\right)\right)\right\rangle,
	\end{align*}
	where the second equality follows from the fact that
	\begin{equation*}\label{recur}
	\bar{x}_{k+1}=\bar{x}_{k}-\alpha_k\left(\frac{\mathbf{u}^\intercal}{n}\otimes\mathbf{I}_{d}\right)\mathbf{y}_k=\bar{x}_{k}-\alpha_k\left(\frac{\mathbf{u}^\intercal}{n}\otimes\mathbf{I}_{d}\right)\left(\mathbf{y}_k^{'}+\xi_k\right).
	\end{equation*}
	Take expectation on both sides of above inequality,
	\begin{align}
	\mathbb{E}\left[h(\bar{x}_{k+1})\right]
	&\le \mathbb{E}\left[h(\bar{x}_k)\right]-\frac{\mathbf{u}^\intercal \mathbf{v}\alpha_{k}}{n}\mathbb{E}\left[\|\nabla h(\bar{x}_k)\|^2\right]+\frac{L}{2}\mathbb{E}\left[\left\|\alpha_{k}\left(\frac{\mathbf{u}^\intercal}{n}\otimes\mathbf{I}_{d}\right)\mathbf{y}_k\right\|^2\right]\notag\\
	\label{grad-bound-1}&\quad+\mathbb{E}\left[\left\langle \nabla h(\bar{x}_k),\alpha_{k}\left(\frac{\mathbf{u}^\intercal \mathbf{v}}{n}\nabla h(\bar{x}_k)-\left(\frac{\mathbf{u}^\intercal}{n}\otimes\mathbf{I}_{d}\right)\left(\mathbf{y}_k^{'}+\xi_k\right)\right)\right\rangle\right].
	\end{align}
	
	For the third term on the right hand of (\ref{grad-bound-1}),
	\begin{align}\label{grad-bound-2}
	\frac{L}{2}\mathbb{E}\left[\left\|\alpha_{k}\left(\frac{\mathbf{u}^\intercal}{n}\otimes\mathbf{I}_{d}\right)\mathbf{y}_k\right\|^2\right]\le \frac{L\alpha_{k}^2\|\mathbf{u}\|^2}{2n^2}\mathbb{E}\left[\left\|\mathbf{y}_k\right\|^2\right]\le \frac{L\|\mathbf{u}\|^2c_b^2\ C_fC_g}{2n^2(1-\tau_\Z)^2}\alpha_{k}^2,
	\end{align}
	where  the second inequalities follows from (\ref{y-bound}).

For the fourth term on the right hand of (\ref{grad-bound-1}),
	\begin{align}
	&\mathbb{E}\left[\left\langle \nabla h(\bar{x}_k),\alpha_{k}\left(\frac{\mathbf{u}^\intercal \mathbf{v}}{n}\nabla h(\bar{x}_k)-\left(\frac{\mathbf{u}^\intercal}{n}\otimes\mathbf{I}_{d}\right)\left(\mathbf{y}_k^{'}+\xi_k\right)\right)\right\rangle\right]\notag\\
	&\le \frac{\alpha_k^2}{2\tau}\mathbb{E}\left[\|\nabla h(\bar{x}_k)\|^2\right]+\frac{3\tau}{2}\left(\frac{\mathbf{u}^\intercal \mathbf{v}}{n}\right)^2\mathbb{E}\left[\|P_1\|^2\right]+\frac{3\tau}{2}\left(\frac{\mathbf{u}^\intercal \mathbf{v}}{n}\right)^2\mathbb{E}\left[\|P_2\|^2\right]\notag\\
	&\quad+\frac{3\tau\|\mathbf{u}\|^2}{2n^2}\mathbb{E}\left[\|\mathbf{v}\otimes\bar{y}_k^{'}-\mathbf{y}_k^{'}\|^2\right]+\frac{\alpha_{k}\|\mathbf{u}\|}{n}\left|\mathbb{E}\left[\left\langle \nabla h(\bar{x}_k),-\alpha_{k}\left(\frac{\mathbf{u}^\intercal}{n}\otimes\mathbf{I}_{d}\right)\xi_k\right\rangle\right]\right|\notag\\
	&\le \frac{\alpha_k^2}{2\tau}\mathbb{E}\left[\|\nabla h(\bar{x}_k)\|^2\right]+\frac{3\tau L^2}{2n}\left(\frac{\mathbf{u}^\intercal \mathbf{v}}{n}\right)^2\mathbb{E}\left[\|\mathbf{x}_k-\mathbf{1}\otimes\bar{x}_k\|^2\right]+\frac{3\tau C_gL_f^2}{2n}\left(\frac{\mathbf{u}^\intercal \mathbf{v}}{n}\right)^2\mathbb{E}\left[\|\mathbf{g}_k-\mathbf{z}_k\|^2\right]\notag\\
	&\quad+\frac{3\tau\|\mathbf{u}\|^2}{2n^2}\mathbb{E}\left[\|\mathbf{v}\otimes\bar{y}_k^{'}-\mathbf{y}_k^{'}\|^2\right]
	+\frac{\alpha_{k}\|\mathbf{u}\|}{n}\left|\mathbb{E}\left[\left\langle \nabla h(\bar{x}_k),-\alpha_{k}\left(\frac{\mathbf{u}^\intercal}{n}\otimes\mathbf{I}_{d}\right)\xi_k\right\rangle\right]\right|\notag\\
	\label{grad-bound-3}&\le \frac{\alpha_k^2}{2\tau}\mathbb{E}\left[\|\nabla h(\bar{x}_k)\|^2\right]+\frac{3\tau L^2U_1}{2n}\alpha_k^2+\frac{3\tau C_gL_f^2}{2n}\mathbb{E}\left[\|\mathbf{g}_k-\mathbf{z}_k\|^2\right]+\frac{3\tau\|\mathbf{u}\|^2U_1}{2n^2}\alpha_k^2+U_3\alpha_k^2,
	\end{align}
	where  $P_1=\nabla h(\bar{x}_k)-\frac{1}{n}\sum_{j=1}^n\nabla g_j(x_{j,k})\nabla f_j(g_j(x_{j,k}))$, $P_2=\frac{1}{n}\sum_{j=1}^n\nabla g_j(x_{j,k})\nabla f_j(g_j(x_{j,k}))-\bar{y}_k^{'}$
	and $\tau$ can be any positive scalar, the first inequality follows from Cauchy-Schwartz inequality and the fact $ab\le \frac{1}{2\tau}a^2+\frac{\tau}{2}b^2$, the second inequality follows from the Lipschitz continuity of $\nabla g_j(\cdot)\nabla f_j(g_j(\cdot))$, Assumption \ref{ass-objective} and the fact $\bar{y}_k^{'}=\frac{1}{n}\sum_{j=1}^n\nabla g_j(x_{j,k})\nabla f_j(z_{j,k})$, the third inequality follows from  (\ref{x-y-consensus}), the fact $\mathbf{u}^\intercal \mathbf{v}\le n$ and Lemma  \ref{lem:nonconv} (ii).

Plug (\ref{grad-bound-2})-(\ref{grad-bound-3}) into (\ref{grad-bound-1}) and set $\tau=\frac{2\alpha_k}{3}$,
	\begin{align}
	\mathbb{E}\left[h(\bar{x}_{k+1})\right]
	&\le \mathbb{E}\left[h(\bar{x}_k)\right]-\alpha_{k}\left(1-\frac{\alpha_k}{2\tau}\right)\mathbb{E}\left[\|\nabla h(\bar{x}_k)\|^2\right]+\frac{L\|\mathbf{u}\|^2c_b^2\ C_fC_g}{2n^2(1-\tau_\Z)^2}\alpha_{k}^2+\frac{3\tau L^2U_1}{2n}\alpha_k^2\notag\\
	&\quad+\frac{3\tau C_gL_f^2}{2n}\mathbb{E}\left[\|\mathbf{g}_k-\mathbf{z}_k\|^2\right]+\frac{3\tau\|\mathbf{u}\|^2U_1}{2n^2}\alpha_k^2+U_3\alpha_k^2\notag\\
	\label{grad-bound-4}&\le \mathbb{E}\left[h(\bar{x}_k)\right]-\frac{\alpha_k}{4}\mathbb{E}\left[\|\nabla h(\bar{x}_k)\|^2\right]+\left(\frac{L\|\mathbf{u}\|^2c_b^2\ C_fC_g}{2n^2(1-\tau_\Z)^2}+U_3\right)\alpha_{k}^2\notag\\
	&\quad+\left(\frac{L^2U_1}{n}+\frac{\|\mathbf{u}\|^2U_1}{n^2}\right)\alpha_k^3+\beta_k\mathbb{E}\left[\|\mathbf{g}_k-\mathbf{z}_k\|^2\right].
	\end{align}

	Combining (\ref{z-convergence}) with (\ref{grad-bound-4}),
	{\small\begin{align}
	&\mathbb{E}\left[h(\bar{x}_{k+1})\right]+\mathbb{E}\left[\|\mathbf{z}_{k+1}- \mathbf{g}_{k+1}\|^2\right]\notag\\
	&\le \mathbb{E}\left[h(\bar{x}_k)\right]+\left[(1-\beta_k)^2+\beta_k\right]\mathbb{E}\left[\|\mathbf{z}_k-\mathbf{g}_k\|^2\right]-\frac{\alpha_k}{4}\mathbb{E}\left[\|\nabla h(\bar{x}_k)\|^2\right]+\left(\frac{L\|\mathbf{u}\|^2c_b^2\ C_fC_g}{2n^2(1-\tau_\Z)^2}+U_3\right)\alpha_{k}^2\notag\\
	 &\quad+\left(\frac{L^2U_1}{n}+\frac{\|\mathbf{u}\|^2U_1}{n^2}\right)\alpha_k^3+\left(12C_g\X^2\normm{\tilde{\mathbf{A}}-\mathbf{I}_{nd}}^2U_1+\frac{12c_b^2nC_g^2C_f\normm{\mathbf{A}}^2}{(1-\tau_\Z)^2}\right)\alpha_k^2+3V_g\beta_k^2\notag\\
	\label{grad-bound-5}&\le \mathbb{E}\left[h(\bar{x}_k)\right]+\mathbb{E}\left[\|\mathbf{z}_k-\mathbf{g}_k\|^2\right]-\frac{\alpha_k}{4}\mathbb{E}\left[\|\nabla h(\bar{x}_k)\|^2\right]+U_4\alpha_{k}^2+\left(\frac{L^2U_1}{n}+\frac{\|\mathbf{u}\|^2U_1}{n^2}\right)\alpha_k^3,
	\end{align}}
	where
	\begin{equation*}
	U_4=\frac{L\|\mathbf{u}\|^2c_b^2\ C_fC_g}{2n^2(1-\tau_\Z)^2}+U_3+12C_g\X^2\normm{\tilde{\mathbf{A}}-\mathbf{I}_{nd}}^2U_1+\frac{12c_b^2nC_g^2C_f\normm{\mathbf{A}}^2}{(1-\tau_\Z)^2}+3V_g\frac{ C_g^2L_f^4}{n^2}.
	\end{equation*}
	Reordering the terms of (\ref{grad-bound-5}) and summing over $k$ from 1 to $K$,
	\begin{align*}
	\sum_{k=1}^K\frac{\alpha_k}{4}\mathbb{E}\left[\|\nabla h(\bar{x}_k)\|^2\right]
    &\le \mathbb{E}\left[h(\bar{x}_1)\right]+\mathbb{E}\left[\|\mathbf{z}_1-\mathbf{g}_1\|^2\right]-\left(\mathbb{E}\left[h(\bar{x}_{K+1})\right]+\mathbb{E}\left[\|\mathbf{z}_{K+1}- \mathbf{g}_{K+1}\|^2\right]\right)\\
    &\quad+U_4\sum_{k=1}^K\alpha_{k}^2+\left(\frac{L^2U_1}{n}+\frac{\|\mathbf{u}\|^2U_1}{n^2}\right)\sum_{k=1}^K\alpha_k^3\\
    &\le \mathbb{E}\left[h(\bar{x}_1)\right]+\mathbb{E}\left[\|\mathbf{z}_1-\mathbf{g}_1\|^2\right]+U_4\sum_{k=1}^K\alpha_{k}^2+\left(\frac{L^2U_1}{n}+\frac{\|\mathbf{u}\|^2U_1}{n^2}\right)\sum_{k=1}^K\alpha_k^3.
	\end{align*}
	Multiplying both sides of the above inequality by $\frac{4}{a\sqrt{K}}$,
	\begin{align*}
	&\frac{1}{K}\sum_{k=1}^K\mathbb{E}\left[\|\nabla h(\bar{x}_k)\|^2\right]\le \frac{4\left(\mathbb{E}\left[h(\bar{x}_1)\right]+\mathbb{E}\left[\|\mathbf{z}_1-\mathbf{g}_1\|^2\right]\right)/a+4a U_4}{\sqrt{K}}+\frac{4\left(\frac{L^2U_1}{n}+\frac{\|\mathbf{u}\|^2U_1}{n^2}\right)a^2}{K}.
	\end{align*}
	By the Lipschitz continuity of $\nabla h(\cdot)$, we have
	{\small\begin{align*}
   \frac{1}{K}\sum_{k=1}^K\mathbb{E}\left[\|\nabla h(x_{i,k})\|^2\right]
	&\le  \frac{2}{K}\sum_{k=1}^K\mathbb{E}\left[\|\nabla h(\bar{x}_k)\|^2\right]+\frac{2L^2}{K}\sum_{k=1}^K\mathbb{E}\left[\|x_{i,k}-\bar{x}_k\|^2\right]\\
	&\le  \frac{8\left(\mathbb{E}\left[h(\bar{x}_1)\right]+\mathbb{E}\left[\|\mathbf{z}_1-\mathbf{g}_1\|^2\right]\right)/a+8a U_4}{\sqrt{K}}+\frac{2\left(\frac{4L^2U_1}{n}+\frac{4\|\mathbf{u}\|^2U_1}{n^2}+L^2U_1\right)a^2}{K},
	\end{align*}}
	where the last inequality follows from (\ref{x-y-consensus}). The proof is complete.
\end{proof}
Theorem \ref{thm:nonconv} presents that the AB-DSCSC achieves the convergence rate $\mathcal{O}(K^{-1/2})$   finding the ($\epsilon$)-stationary point,  which is same as the convergence rate of stochastic gradient descent for non-compositional problems.
On the other hand,  the sample complexity  for  finding the ($\epsilon$)-stationary point is  $\mathcal{O}\left(\frac{1}{\epsilon^2}\right)$  as AB-DSCSC  does not need the increasing batch size strategy in each iteration.

Next, we study the convergence rate of AB-DSCSC for the   strongly convex objective under diminishing stepsize strategy.
\begin{thm}\label{thm:rate}
Let $\alpha_k=a/(k+b)^\alpha,~\beta_k=\beta\alpha_k$, where $a>0,b\ge0$, $\alpha\in (1/2, 1)$, $\beta\in (0, 1/a)$ and $a/(1+b)^\alpha\le \frac{n}{\mathbf{u}^\intercal \mathbf{v}\mu}\min\{1,2/(C_g^2L_f + C_fL_g)\}$. Under Assumptions \ref{ass-objective}-\ref{ass:matrix} and the condition that objective function $h(x)$ is $\mu$-strongly convex,
\begin{equation*}
\mathbb{E}\left[\|\bar{x}_k-x^*\|^2\right]=\mathcal{O}\left(\alpha_k\right).
\end{equation*}
Moreover, if $\alpha_{k}=a/(k+b)$, $\beta_k=\beta\alpha_{k}$,
$\frac{2n}{\mathbf{u}^\intercal \mathbf{v}\mu}<a\le \frac{n(b+1)}{\mathbf{u}^\intercal \mathbf{v}\mu}\min\left\{1,2/(C_g^2L_f + C_fL_g)\right\}$ and $1<\beta a\le 1+b$,
$$\mathbb{E}\left[\|\bar{x}_k-x^*\|^2\right]=\mathcal{O}\left(\frac{1}{k}\right).$$
\end{thm}
\begin{proof}
Recall the definition $\bar{x}_{k+1}= \left(\frac{\mathbf{u}^\intercal}{n}\otimes\mathbf{I}_{d}\right)\mathbf{x}_{k+1}$ in Lemma \ref{lem:rate},
\begin{align}
\bar{x}_{k+1}
&=\left(\frac{\mathbf{u}^\intercal}{n}\otimes\mathbf{I}_{d}\right)\tilde{\mathbf{A}}\left(\mathbf{x}_k-\alpha_k\mathbf{y}_k\right)\notag\\
&=\bar{x}_{k}-\alpha_k\left(\frac{\mathbf{u}^\intercal}{n}\otimes\mathbf{I}_{d}\right)\left(\mathbf{y}_k^{'}+\xi_k\right)\notag\\
&=\bar{x}_{k}-\frac{\mathbf{u}^\intercal \mathbf{v}\alpha_{k}}{n}\nabla h(\bar{x}_k)+\frac{\mathbf{u}^\intercal \mathbf{v}\alpha_{k}}{n}\Bigg(\underbrace{\nabla h(\bar{x}_k)-\frac{1}{n}\sum_{j=1}^n\nabla g_j(x_{j,k})\nabla f_j(g_j(x_{j,k}))}_{P^{(1)}_k}\notag\\
&\quad+\underbrace{\frac{1}{n}\sum_{j=1}^n\nabla g_j(x_{j,k})\nabla f_j(g_j(x_{j,k}))-\bar{y}^{'}_k}_{P^{(2)}_k}+\underbrace{\frac{n}{\mathbf{u}^\intercal \mathbf{v}}\left(\frac{\mathbf{u}^\intercal}{n}\otimes\mathbf{I}_{d}\right)\left(\mathbf{v}\otimes\bar{y}^{'}_k-\mathbf{y}_k^{'}\right)}_{P^{(3)}_k}\notag\\
\label{consensus-new-form}&\quad+\underbrace{\left(-\frac{n}{\mathbf{u}^\intercal \mathbf{v}}\right)\left(\frac{\mathbf{u}^\intercal}{n}\otimes\mathbf{I}_{d}\right)\xi_k}_{P^{(4)}_k}\Bigg),
\end{align}
where $\mathbf{y}_k^{'}$ and $\xi_{k+1}$ are defined in (\ref{g-tra}) and (\ref{noise}), the second equality follows from the fact $\mathbf{u}^\intercal\mathbf{A}=\mathbf{1}$. Subsequently,
	\begin{align}
	&\mathbb{E}\left[\|\bar{x}_{k+1}-x^*\|^2\right]\notag\\
	&=\mathbb{E}\left[\left\|\bar{x}_{k}-x^*-\frac{\mathbf{u}^\intercal \mathbf{v}\alpha_{k}}{n}\nabla h(\bar{x}_k)\right\|^2\right]+\left(\frac{\mathbf{u}^\intercal \mathbf{v}\alpha_{k}}{n}\right)^2\mathbb{E}\left[\left\|P^{(1)}_k+P^{(2)}_k+P^{(3)}_k+P^{(4)}_k\right\|^2\right]\notag\\
	&\quad+2\left(\frac{\mathbf{u}^\intercal \mathbf{v}\alpha_{k}}{n}\right)\mathbb{E}\left[\left\langle \bar{x}_{k}-x^*-\frac{\mathbf{u}^\intercal \mathbf{v}\alpha_{k}}{n}\nabla h(\bar{x}_k), P^{(1)}_k+P^{(2)}_k+P^{(3)}_k+P^{(4)}_k\right\rangle\right]\notag\\
	&\le\left(1-\frac{\mathbf{u}^\intercal \mathbf{v}\mu\alpha_{k}}{n}\right)^2\mathbb{E}\left[\left\|\bar{x}_{k}-x^*\right\|^2\right]+\left(\frac{\mathbf{u}^\intercal \mathbf{v}\alpha_{k}}{n}\right)^2\mathbb{E}\left[\left\|P^{(1)}_k+P^{(2)}_k+P^{(3)}_k+P^{(4)}_k\right\|^2\right]\notag\\
	&\quad+2\left(\frac{\mathbf{u}^\intercal \mathbf{v}\alpha_{k}}{n}\right)\mathbb{E}\left[\left\langle \bar{x}_{k}-x^*-\frac{\mathbf{u}^\intercal \mathbf{v}\alpha_{k}}{n}\nabla h(\bar{x}_k), P^{(1)}_k+P^{(2)}_k+P^{(3)}_k+P^{(4)}_k\right\rangle\right]\notag\\
	&\le\left(\left(1-\frac{\mathbf{u}^\intercal \mathbf{v}\mu\alpha_{k}}{n}\right)^2+\frac{\tau}{2}\left(\frac{\mathbf{u}^\intercal \mathbf{v}\alpha_{k}}{n}\right)^2L^2 \right)\mathbb{E}\left[\left\|\bar{x}_{k}-x^*\right\|^2\right]\notag\\
	&\quad+\left(1+\frac{1}{2\tau}\right)\left(\frac{\mathbf{u}^\intercal \mathbf{v}\alpha_{k}}{n}\right)^2\mathbb{E}\left[\left\|P^{(1)}_k+P^{(2)}_k+P^{(3)}_k+P^{(4)}_k\right\|^2\right]\notag\\
	\label{eq-auxi-3}&\quad+2\left(\frac{\mathbf{u}^\intercal \mathbf{v}\alpha_{k}}{n}\right)\mathbb{E}\left[\left\langle \bar{x}_{k}-x^*,P^{(1)}_k+P^{(2)}_k+P^{(3)}_k+P^{(4)}_k\right\rangle\right],
	\end{align}
where $\tau$ is any positive scalar, the first inequality follows from \cite[Lemm 10]{qu2017harnessing}, the second inequalities follows from the inequality $ab\le \frac{\tau a^2}{2}+\frac{b^2}{2\tau}$ and the fact that $\nabla h(x)$ is $L(\define C_g^2L_f + C_fL_g)$-smooth.

 For the second term on the right hand side of (\ref{eq-auxi-3}),
{\small
\begin{align}
&\left(1+\frac{1}{2\tau}\right)\left(\frac{\mathbf{u}^\intercal \mathbf{v}\alpha_{k}}{n}\right)^2\mathbb{E}\left[\left\|P^{(1)}_k+P^{(2)}_k+P^{(3)}_k+P^{(4)}_k\right\|^2\right]\notag\\
&\le \left(1+\frac{1}{2\tau}\right)\left(4\left(\frac{\mathbf{u}^\intercal \mathbf{v}\alpha_{k}}{n}\right)^2\frac{L^2\X^2}{n}\mathbb{E}\left[\|x_{k}-\mathbf{1}\otimes\bar{x}_{k}\|_\A^2\right]+4\left(\frac{\mathbf{u}^\intercal \mathbf{v}\alpha_{k}}{n}\right)^2\frac{C_g^2L_f^2}{n}\mathbb{E}\left[\left\|\mathbf{g}_k-\mathbf{z}_k\right\|^2\right]\right.\notag\\
&\quad\left.+4\alpha_k^2\frac{\|\mathbf{u}\|^2}{n^2}\X^2\mathbb{E}\left[\left\|\mathbf{y}_k^{'}-\mathbf{v}\otimes\bar{y}^{'}_k\right\|_\Z^2\right]+4\frac{\|\mathbf{u}\|^2}{n^2}\alpha_k^2\mathbb{E}\left[\left\|\xi_k\right\|^2\right]\right)\notag\\
&\le \left(1+\frac{1}{2\tau}\right)\left(4\left(\frac{\mathbf{u}^\intercal \mathbf{v}\alpha_{k}}{n}\right)^2\frac{L^2\X^2U_1\alpha_k^2}{n}+4\left(\frac{\mathbf{u}^\intercal \mathbf{v}\alpha_{k}}{n}\right)^2\frac{C_g^2L_f^2}{n}\mathbb{E}\left[\left\|\mathbf{g}_k-\mathbf{z}_k\right\|^2\right]+4\frac{\|\mathbf{u}\|^2}{n^2}\X^2U_1\alpha_k^3\right)\notag\\
\label{error-1}&\quad+\left(1+\frac{1}{2\tau}\right)4\frac{\|\mathbf{u}\|^2}{n^2}\frac{c_b^24n C_fC_g}{(1-\tau_\Z)^2}\alpha_k^2,
\end{align}
}
where $c_b=\max\left\{\X,\frac{\normm{\mathbf{B}-\mathbf{I}_{n}}_\Z}{\tau_\Z}\X\right\}$, the first inequality follows from Assumption \ref{ass-objective}(c) and the Lipschitz continuity of  $\nabla f_j(\cdot)$, the second inequality follows from (\ref{x-y-consensus}) and Lemma \ref{lem:nonconv}. In addition,  by Lemma \ref{lem:g-tra} and \cite[Lemmas 4-5 in Chapter 2]{polyak1987Introduction}, there exists a constant $U_2$ such that
\begin{equation}\label{z-bound-2}
\mathbb{E}\left[\left\|\mathbf{g}_k-\mathbf{z}_k\right\|^2\right]\le U_2\beta_k=U_2\beta \alpha_k.
\end{equation}

Combining (\ref{error-1}) with (\ref{z-bound-2}), we have
{\small
	\begin{align}
	&\left(1+\frac{1}{2\tau}\right)\left(\frac{\mathbf{u}^\intercal \mathbf{v}\alpha_{k}}{n}\right)^2\mathbb{E}\left[\left\|P^{(1)}_k+P^{(2)}_k+P^{(3)}_k+P^{(4)}_k\right\|^2\right]\notag\\
	&\le \left(1+\frac{1}{2\tau}\right)\left(4\left(\frac{\mathbf{u}^\intercal \mathbf{v}\alpha_{k}}{n}\right)^2\frac{L^2\X^2U_1\alpha_k^2}{n}+4\left(\frac{\mathbf{u}^\intercal \mathbf{v}\alpha_{k}}{n}\right)^2\frac{C_g^2L_f^2}{n}U_2\beta\alpha_k+4\frac{\|\mathbf{u}\|^2}{n^2}\X^2U_1\alpha_k^3\right)\notag\\
	\label{error-3-0}&\quad+\left(1+\frac{1}{2\tau}\right)4\frac{\|\mathbf{u}\|^2}{n^2}\frac{c_b^24n C_fC_g}{(1-\tau_\Z)^2}\alpha_k^2\\
	\label{error-3}&\le 16\left(1+\frac{\left(C_g^2L_f + C_fL_g\right)}{\mu^2}\right)\frac{\|\mathbf{u}\|^2}{n}\frac{c_b^2 C_fC_g}{(1-\tau_\Z)^2}\alpha_k^2+o(\alpha_k^2),
	\end{align}}
where
\begin{equation}\label{tau}
\tau=\frac{\mu^2}{2\left(C_g^2L_f + C_fL_g\right)}.
\end{equation}

For the third term on the right hand side of (\ref{eq-auxi-3}),
\begin{align}
&2\left(\frac{\mathbf{u}^\intercal \mathbf{v}\alpha_{k}}{n}\right)\mathbb{E}\left[\left\langle \bar{x}_{k}-x^*,P^{(1)}_k+P^{(2)}_k+P^{(3)}_k+P^{(4)}_k\right\rangle\right]\notag\\
&\le \tau_1 \mathbb{E}\left[\left\|\bar{x}_{k}-x^*\right\|^2\right]+\frac{1}{\tau_1}\left(\frac{\mathbf{u}^\intercal \mathbf{v}\alpha_{k}}{n}\right)^2\mathbb{E}\left[\left\|P^{(1)}_k+P^{(2)}_k+P^{(3)}_k\right\|^2\right]+2\left(\frac{\mathbf{u}^\intercal \mathbf{v}\alpha_{k}}{n}\right)\mathbb{E}\left[\left\langle \bar{x}_{k}-x^*,P^{(4)}_k\right\rangle\right]\notag\\
&\le \tau_1 \mathbb{E}\left[\left\|\bar{x}_{k}-x^*\right\|^2\right]+\frac{12}{\tau_1}\left(\frac{\mathbf{u}^\intercal \mathbf{v}\alpha_{k}}{n}\right)^2\frac{L^2\X^2U_1\alpha_k^2}{n}+\frac{12}{\tau_1}\left(\frac{\mathbf{u}^\intercal \mathbf{v}\alpha_{k}}{n}\right)^2\frac{C_g^2L_f^2}{n}U_2\beta\alpha_k+\frac{12}{\tau_1}\frac{\|\mathbf{u}\|^2}{n^2}\X^2U_1\alpha_k^3\notag\\
&\quad+\frac{\|\mathbf{u}\|^2c_bc_0}{n^2(1-\tau_\Z)}\left(\frac{c_b^2nC_gC_f}{(1-\tau_\Z)^2}+4nC_g C_f\right)\alpha_k^3\notag\\
\label{error-2}&\le \frac{\mathbf{u}^\intercal \mathbf{v}\mu\alpha_k}{4n} \mathbb{E}\left[\left\|\bar{x}_{k}-x^*\right\|^2\right]+\left(\frac{48\mathbf{u}^\intercal \mathbf{v}C_g^2L_f^2\mu U_2\beta}{n^2}
+\frac{48\|\mathbf{u}\|^2\X^2U_1}{n\mathbf{u}^\intercal \mathbf{v}\mu}\right)\alpha_k^2+o(\alpha_k^2),
\end{align}
where $c_0$ is some constant scalar,
\begin{equation}\label{tau-1}
\tau_1=\frac{\mathbf{u}^\intercal \mathbf{v}\mu}{4n}\alpha_k,
\end{equation}
the first inequality follows from the fact $ab\le \frac{\tau_1 a^2}{2}+\frac{b^2}{2\tau_1}$ for any positive scalar $\tau_1$, the second inequality follows from (\ref{error-3-0}) and Lemma \ref{lem:noi-bound} in Appendix.

 Substitute (\ref{error-3})-(\ref{tau-1}) into (\ref{eq-auxi-3}),
\begin{align*}
\mathbb{E}\left[\|\bar{x}_{k+1}-x^*\|^2\right]
&\le \left(1-\frac{\mathbf{u}^\intercal \mathbf{v}\mu\alpha_{k}}{2n}\right)\mathbb{E}\left[\left\|\bar{x}_{k}-x^*\right\|^2\right]+o\left(\alpha_k^2\right)+\left[\frac{48\mathbf{u}^\intercal \mathbf{v}C_g^2L_f^2\mu U_2\beta}{n^2}
+\frac{48\|\mathbf{u}\|^2\X^2U_1}{n\mathbf{u}^\intercal \mathbf{v}\mu}\right.\notag\\
&\quad \left.+16\left(1+\frac{\left(C_g^2L_f + C_fL_g\right)}{\mu^2}\right)\frac{\|\mathbf{u}\|^2}{n}\frac{c_b^2 C_fC_g}{(1-\tau_\Z)^2}\right]\alpha_k^2.
\end{align*}
Then by \cite[Lemmas 4-5 in Chapter 2]{polyak1987Introduction}, $$\mathbb{E}\left[\|\bar{x}_{k+1}-x^*\|^2\right]=\mathcal{O}\left(\alpha_k\right)~ \text{if}~\alpha_k=a/(k+b)^\alpha,\alpha\in (1/2,1),$$
and
$$ \mathbb{E}\left[\|\bar{x}_k-x^*\|^2\right]=\mathcal{O}\left(\frac{1}{k}\right)~ \text{if}~\alpha_k=a/(k+b),a>\frac{2n}{\mathbf{u}^\intercal \mathbf{v}\mu}.$$
The proof is complete.
\end{proof}

Theorem \ref{thm:rate} shows that AB-DSCSC achieves the convergence rate $\mathcal{O}\left(\frac{1}{k}\right)$ for finding the optimal solution, which is also the optimal convergence rate of stochastic gradient descent for non-compositional stochastic strongly convex optimization \cite{Rakhlin2012making}.

The next theorem studies the asymptotic normality of AB-DSCSC.
\begin{thm}\label{thm:asym-norm}
Let stepsizes $\alpha_k=a/(k+b)^\alpha,~\beta_k=\beta\alpha_k$, where $a>0,b\ge0$, $\alpha\in (1/2, 1)$, $\beta\in (0, 1/a)$ and $a/(1+b)^\alpha\le \frac{n}{\mathbf{u}^\intercal \mathbf{v}\mu}\min\{1,2/(C_g^2L_f + C_fL_g)\}$. Suppose

\begin{itemize}
	\item[(a)] Assumptions \ref{ass-objective}-\ref{ass:matrix} hold;
	\item [(b)]$h(x)$ is $\mu$-strongly convex;
	\item [(c)]there exist scalar $C$ and matrix $\mathbf{H}$ such that
	\begin{equation*}\label{second order massage}
	\left\|\nabla h(x)-\frac{1}{n}\mathbf{H}(x-x^*)\right\|\le C\|x-x^*\|^{1+\gamma},\quad \forall x\in \mathbb{R}^{d},
	\end{equation*}
	where $\gamma\in (0,1]$ satisfies that $\sum_{k=1}^\infty\frac{\alpha_k^{(1+\gamma)/2}}{\sqrt{k}}<\infty$;
	\item [(d)] for any $i\in\mathcal{V}$, there exist scalar $C_i$ and matrix  $\mathbf{T}_i$ such that
	\begin{equation*}\label{second order massage-1}
	\left\|\nabla f_i(y)-\nabla f_i(y^{'})-\mathbf{T}_i\left(y-y^{'}\right)\right\|\le C_i\|y-y^{'}\|^{1+\gamma},\quad \forall y,y^{'}\in \mathbb{R}^p, 
	\end{equation*}
	\item [(e)] for any $i\in\mathcal{V}$, $G_i(\cdot;\phi)$ is Lipschitz continuous with coefficient $L_g^{'}$ , that is
	\begin{equation*}
	\left\|G_i(x;\phi)-G_i(x^{'};\phi)\right\|\le L_g^{'}\|x-x^{'}\|,\quad \forall y,y^{'}\in \mathbb{R}^p.
	\end{equation*}
\end{itemize}
Then for any $i\in\mathcal{V}$,
{\small\begin{equation}\label{limit distribution}
\frac{1}{\sqrt{k}}\sum_{t=1}^{k}
\left(
\begin{array}{c}
x_{i,t}-x^*\\
\frac{\sum\limits_{j=1}^{n}\nabla g_j(x^*)\mathbf{T}_j\left(z_{j,t}-g_j\left(x_{j,t}\right)\right)}{n}
\end{array}
\right)\stackrel{d}{\longrightarrow} N\left(\mathbf{0},\left(
\begin{array}{cc}
\mathbf{H}^{-1}\left(\mathbf{S}_1+\mathbf{S}_2\right)(\mathbf{H}^{-1})^\intercal& -\frac{1}{n}\mathbf{H}^{-1}\mathbf{S}_2\\
-\frac{1}{n}\mathbf{S}_2(\mathbf{H}^{-1})^\intercal& \frac{1}{n^2}\mathbf{S}_2
\end{array}
\right)\right), 
\end{equation}}
where $\mathbf{S}_1=\Cov\left(\nabla G_j(x^*;\phi_j)\nabla F_j(g(x^*);\zeta_j)\right)$, $\mathbf{S}_2=\Cov\left(\sum_{j=1}^n\nabla g_j(x^*)\mathbf{T}_jG_j(x^*;\phi_j)\right)$.
\end{thm}
\begin{proof}
By (\ref{x-y-consensus}),
\begin{equation*}
\begin{aligned}
&\mathbb{E}\left[\left\|\frac{1}{\sqrt{k}}\sum_{t=0}^{k-1}\left(\bar{x}_{t}-x^*\right)-\frac{1}{\sqrt{k}}\sum_{t=0}^{k-1}\left(x_{i,t}-x^*\right)\right\|\right]\\
&\le \frac{1}{\sqrt{k}}\sum_{t=0}^{k-1}\sqrt{\mathbb{E}\left[\|x_{t}-\mathbf{1}\otimes\bar{x}_{t}\|^2\right]}\le\frac{\sqrt{U_1}}{\sqrt{k}}\sum_{t=0}^{k-1}\alpha_t\rightarrow 0.
\end{aligned}
\end{equation*}
Then by Slutsky's theorem, it  is sufficient  to show
{\small\begin{equation*}
\frac{1}{\sqrt{k}}\sum_{t=1}^{k}
\left(
\begin{array}{c}
\bar{x}_t-x^*\\
\frac{\sum_{j=1}^{n}\nabla g_j(x^*)\mathbf{T}_j\left(z_{j,t}-g_j\left(x_{j,t}\right)\right)}{n}
\end{array}
\right)\stackrel{d}{\longrightarrow} N\left(\mathbf{0},\left(
\begin{array}{cc}
\mathbf{H}^{-1}\left(\mathbf{S}_1+\mathbf{S}_2\right)(\mathbf{H}^{-1})^\intercal& -\frac{1}{n}\mathbf{H}^{-1}\mathbf{S}_2\\
-\frac{1}{n}\mathbf{S}_2(\mathbf{H}^{-1})^\intercal& \frac{1}{n^2}\mathbf{S}_2
\end{array}
\right)\right).
\end{equation*}}

Subtract $x^*$ from both sides of (\ref{consensus-new-form}),
\begin{align}
\bar{x}_{k+1}-x^*
&=\bar{x}_{k}-x^*-\frac{\mathbf{u}^\intercal \mathbf{v}\alpha_{k}}{n}\nabla h(\bar{x}_k)+\left(\frac{\mathbf{u}^\intercal \mathbf{v}\alpha_{k}}{n}\right)\left(P^{(1)}_k+P^{(2)}_k+P^{(3)}_k+P^{(4)}_k\right)\notag\\
&=\left(\mathbf{I}_d-\tilde{\alpha}_k\frac{1}{n}\mathbf{H}\right)(\bar{x}_{k}-x^*)-\tilde{\alpha}_k\frac{1}{n}\sum_{j=1}^{n}\nabla g_j(x^*)\mathbf{T}_j\left(z_{j,k}-g_j\left(x_{j,k}\right)\right)\notag\\
\label{consensus-new-form-1}&\quad+\tilde{\alpha}_k\left(P^{(0)}_k+P^{(1)}_k+P^{(3)}_k+P^{(4)}_k\right),
\end{align}
where $\tilde{\alpha}_k=\frac{\mathbf{u}^\intercal \mathbf{v}\alpha_{k}}{n}$, 
\begin{equation}\label{P-0}
P^{(0)}_k=-\left(\nabla h(\bar{x}_k)-\frac{1}{n}\mathbf{H}(\bar{x}_{k}-x^*)\right)+\frac{1}{n}\sum_{j=1}^{n}\nabla g_j(x^*)\mathbf{T}_j\left(z_{j,k}-g_j\left(x_{j,k}\right)\right)+P^{(2)}_k.
\end{equation}
According to recursion (\ref{z-recur}) and the definition of $\beta_k$,
{\small\begin{align*}\label{z-recur-1}
z_{i,k+1}-g_i\left(x_{i,k+1}\right)
=\left(1-\frac{n\beta}{\mathbf{u}^\intercal \mathbf{v}}\tilde{\alpha}_k\right)\left(z_{i,k}-g_i\left(x_{i,k}\right)\right)+G_{i,k+1}^{(1)}-g_i(x_{i,k+1})+\left(1-\beta_k\right)\left(g_i(x_{i,k})-G_{i,k+1}^{(2)}\right),
\end{align*}}
where $G_{i,k+1}^{(1)}=G_i(x_{i,k+1};\phi_{i,k+1}^{'})$, $G_{i,k+1}^{(2)}=G_i(x_{i,k};\phi_{i,k+1}^{'})$.
Combining above equation with (\ref{consensus-new-form-1}),
\begin{equation}\label{recur-2}
\Delta_{k+1}=\left(\mathbf{I}_{2d}-\tilde{\alpha}_k\mathbf{H}_{\theta}\right)\Delta_k+\tilde{\alpha}_k\eta_k^{(1)}+\tilde{\alpha}_k\left(\eta_k^{(2)}+\eta_k^{(3)}\right),
\end{equation}
where
\begin{equation*}
\Delta_k=\left(
\begin{array}{c}
\bar{x}_{k}-x^*\\
\frac{\sum_{j=1}^{n}\nabla g_j(x^*)\mathbf{T}_j\left(z_{j,k}-g_j\left(x_{j,k}\right)\right)}{n}
\end{array}
\right),\quad\mathbf{H}_{\theta}=\left(
\begin{array}{cc}
\frac{1}{n}\mathbf{H}& \mathbf{I}_d\\
\mathbf{0}& \frac{n\beta}{\mathbf{u}^\intercal \mathbf{v}}\mathbf{I}_d
\end{array}
\right),
\end{equation*}
{\small\begin{equation}\label{eta-2}
	\eta_k^{(1)}=\left(
	\begin{aligned}
	&\quad\quad\quad~~\quad\quad\quad~~\quad\quad P^{(4)}_k\\
	&\frac{\beta}{\mathbf{u}^\intercal \mathbf{v}}\sum_{j=1}^{n}\nabla g_j(x^*)\mathbf{T}_j\left(G_j(x^*;\phi_{j,k+1}^{'})-g_j(x^*)\right)
	\end{aligned}
	\right),\quad\eta_k^{(2)}=	\left(
	\begin{array}{c}
	P^{(0)}_k+P^{(1)}_k+P^{(3)}_k\\
	\mathbf{0}
	\end{array}
	\right),
	\end{equation}}
and
{\small\begin{equation*}
	\eta_k^{(3)}=\left(
	\begin{array}{c}
	\mathbf{0}\\
	\sum\limits_{j=1}^{n}\nabla g_j(x^*)\mathbf{T}_j\left(\frac{G_{j,k+1}^{(1)}-g_j(x_{j,k+1})+\left(1-\beta_k\right)\left(g_j(x_{j,k})-G_{j,k+1}^{(2)}\right)}{n\tilde{\alpha}_k}-\frac{\beta}{\mathbf{u}^\intercal \mathbf{v}}\left(G_j(x^*;\phi_{j,k+1}^{'})-g_j(x^*)\right)\right)
	\end{array}
	\right).
	\end{equation*}}

Denote $\mathbf{M}(k,t)=\tilde{\alpha}_t\sum_{l_1=t}^k\Pi_{l_2=t+1}^{l_1}\left(\mathbf{I}_{2d}-\tilde{\alpha}_k\mathbf{H}_{\theta}\right),\quad \mathbf{N}(k,t)=\mathbf{M}(k,t)-\mathbf{H}_{\theta}^{-1}$.
Then by the recursion (\ref{recur-2}),
{\small\begin{align}
	\frac{1}{\sqrt{k}}\sum_{t=1}^{k}\Delta_t&=\frac{1}{\sqrt{k}}\sum_{t=1}^{k}\mathbf{H}_{\theta}^{-1}\eta_t^{(1)}+\frac{1}{\sqrt{k}}\sum_{t=1}^{k}\mathbf{N}(k,t)\eta_t^{(1)}+\frac{1}{\sqrt{k}}\sum_{t=1}^{k}\mathbf{M}(k,t)\eta_t^{(2)}\notag\\
	\label{recur-1}&\quad+\frac{1}{\sqrt{k}}\sum_{t=1}^{k}\mathbf{M}(k,t)\eta_t^{(3)}+\mathcal{O}\left(\frac{1}{\sqrt{k}}\right).
	\end{align}}

It is easy to show that the second term on the right hand side of  (\ref{recur-1}) converge to 0 in probability, see Lemma \ref{lem:con-pro} in Appendix for details. For the third term on the right hand side of  (\ref{recur-1}),
{\small\begin{equation*}
\begin{aligned}
&\mathbb{E}\left[\left\|\frac{1}{\sqrt{k}}\sum_{t=1}^{k}\mathbf{M}(k,t)\eta_t^{(2)}\right\|\right]\\
&\le \frac{1}{\sqrt{k}}\sum_{t=1}^{k}\normm{\mathbf{M}(k,t)}\left(\mathbb{E}\left[\left\|\frac{1}{n}\sum_{j=1}^{n}\nabla g_j(x^*)\left(\nabla f_j(g_j(x_{j,t}))-\nabla f_j(z_{j,t})-\mathbf{T}_j\left(z_{j,t}-g_j\left(x_{j,t}\right)\right)\right)\right\|\right]\right.\\
&\quad\left.+\mathbb{E}\left[\left\|\nabla h(\bar{x}_t)-\frac{1}{n}\mathbf{H}(\bar{x}_{t}-x^*)\right\|\right]+\mathbb{E}\left[\left\|\frac{1}{n}\sum_{j=1}^{n}\left(\nabla g_j(x_{j,t})-\nabla g_j(x^*)\right)\left(\nabla f_j(g_j(x_{j,t}))-\nabla f_j(z_{j,t})\right)\right\|\right]\right)\\
&\quad+\frac{1}{\sqrt{k}}\sum_{t=1}^{k}\normm{\mathbf{M}(k,t)}\mathbb{E}\left[\left\|P^{(1)}_t+P^{(3)}_t\right\|\right]\\
&\le \frac{1}{\sqrt{k}}\sum_{t=1}^{k}\normm{\mathbf{M}(k,t)}\left(\frac{1}{n}\sum_{j=1}^{n}\|\nabla g_j(x^*)\|\mathbb{E}\left[\left\|z_{j,t}-g_j\left(x_{j,t}\right)\right\|^{1+\gamma}\right]+\mathbb{E}\left[\left\|\bar{x}_{t}-x^*\right\|^{1+\gamma}\right]\right.\\
&\quad\left.+\frac{1}{n}\sum_{j=1}^{n}L_gL_f\sqrt{\mathbb{E}\left[\left\|x_{j,t}-x^*\right\|^2\right]\mathbb{E}\left[\left\|g_j(x_{j,t})-z_{j,t}\right\|^2\right]}\right)+\frac{1}{\sqrt{k}}\sum_{t=1}^{k}\normm{\mathbf{M}(k,t)}\mathbb{E}\left[\left\|P^{(1)}_t+P^{(3)}_t\right\|\right]\\
&= \frac{1}{\sqrt{k}}\sum_{t=1}^{k}\normm{\mathbf{M}(k,t)}\mathcal{O}\left(\alpha_t^{(1+\gamma)/2}+\alpha_t\right),
\end{aligned}
\end{equation*}}
where the first inequality follows from the definitions of $\eta_t^{(2)}$, $P_t^{(0)}$ and $P_t^{(2)}$ in (\ref{eta-2}), (\ref{P-0}) and (\ref{consensus-new-form}), the second inequality follows from condition (d), Assumption \ref{ass-objective} (a) and the H\"{o}lder inequality, the equality follows from (\ref{x-y-consensus}), Lemma \ref{lem:g-tra}, Theorem \ref{thm:rate} and
(\ref{error-1}). Then by  the boundedness of $\mathbf{M}(k,t)$ \cite[Lemma 1 (ii)]{Polyak1992}, the fact $\sum_{k=1}^\infty\frac{\alpha_k^{(1+\gamma)/2}}{\sqrt{k}}<\infty$ and Kronecker Lemma, we have
\begin{equation*}
\mathbb{E}\left[\left\|\frac{1}{\sqrt{k}}\sum_{t=1}^{k}\mathbf{M}(k,t)\eta_t^{(2)}\right\|\right]\le  \frac{1}{\sqrt{k}}\sum_{t=1}^{k}\normm{\mathbf{M}(k,t)}\mathcal{O}\left(\alpha_t^{(1+\gamma)/2}+\alpha_t\right)\longrightarrow 0.
\end{equation*}
Noting that ${\eta_k^{(3)}}$ is a martingale difference sequence adapted to the filtration $\mathcal{F}_k$ (\ref{s-algebra}), the fourth term on the right hand side of  (\ref{recur-1})
{\small\begin{equation*}
\begin{aligned}
&\mathbb{E}\left[\left\|\frac{1}{\sqrt{k}}\sum_{t=1}^{k}\mathbf{M}(k,t)\eta_t^{(3)}\right\|^2\right]\\
&=\frac{1}{k}\sum_{t=1}^{k}\mathbb{E}\left[\left\|\mathbf{M}(k,t)\sum_{j=1}^{n}\nabla g_j(x^*)\mathbf{T}_j\left(\frac{G_{j,t+1}^{(1)}-g_j(x_{j,t+1})-\left(G_{j,t+1}^{(2)}-g_j(x_{j,t})\right)}{n\tilde{\alpha}_t}\right.\right.\right.\\
&\quad\left.+\frac{\beta}{\mathbf{u}^\intercal \mathbf{v}}\left(G_{j,t+1}^{(2)}-g_j(x_{j,t})-\left(G_j(x^*;\phi_{j,t+1}^{'})-g_j(x^*)\right)\right)\bigg)\bigg\|^2\right]\\
&\le \frac{1}{k}\sum_{t=1}^{k}\frac{1}{n}\sum_{j=1}^{n}\normm{\mathbf{M}(k,t)}^2\|\nabla g_j(x^*)\|^2\|\mathbf{T}_j\|^24\left(\left(\frac{L_g^{'}}{\tilde{\alpha}_t}\right)^2\mathbb{E}\left[\left\|x_{j,t+1}-x_{j,t}\right\|^2\right]+\left(\frac{n\beta L_g^{'}}{\mathbf{u}^\intercal \mathbf{v}}\right)^2\mathbb{E}\left[\left\|x_{j,t}-x^*\right\|^2\right]\right)\\
&=\frac{1}{k}\sum_{t=1}^{k}\mathcal{O}\left(\alpha_t\right),
\end{aligned}
\end{equation*}}
where the inequality follows from the Lipschitz continuity of $G_j(\cdot;\phi)$, the second equality follows from (\ref{x-y-consensus}), Theorem \ref{thm:rate} and the fact
\begin{equation*}
\mathbb{E}\left[\left\|x_{j,t+1}-x_{j,t}\right\|^2\right]\le 3\left(\mathbb{E}\left[\left\|x_{j,t+1}-\bar{x}_{t+1}\right\|^2\right]+\mathbb{E}\left[\left\|x_{j,t}-\bar{x}_{t}\right\|^2\right]+\mathbb{E}\left[\left\|\bar{x}_{t+1}-\bar{x}_{t}\right\|^2\right]\right)=\mathcal{O}\left(\alpha_t^2\right).
\end{equation*}
Then by Kronecker Lemma, $\mathbb{E}\left[\left\|\frac{1}{\sqrt{k}}\sum_{t=1}^{k}\mathbf{M}(k,t)\eta_t^{(3)}\right\|^2\right]=\frac{1}{k}\sum_{t=1}^{k}\mathcal{O}\left(\alpha_t\right)\longrightarrow 0.$

It is left to show the asymptotic normality of the first term on the right hand side of  (\ref{recur-1}). Indeed, by the similar way to \cite[Lemma 6 in Appendix B]{zhao2021asymptotic}, we may obtain that
\begin{equation*}
\mathbb{E}\left[\left\|\frac{1}{\sqrt{k}}\sum_{t=0}^{k-1} P^{(4)}_k-\frac{1}{\sqrt{k}}\sum_{t=0}^{k-1}\left(\frac{\mathbf{1}^\intercal}{n}\otimes \mathbf{I}_d\right)\epsilon_t^*\right\|^2\right]\longrightarrow 0,\quad \frac{1}{\sqrt{k}}\sum_{t=1}^{k}\left(\frac{\mathbf{1}^\intercal}{n}\otimes \mathbf{I}_d\right)\epsilon_t^*\stackrel{d}{\rightarrow} N\left(\mathbf{0},\frac{1}{n^2}\mathbf{S}_1\right)
\end{equation*}
and 
\begin{equation*}
\frac{1}{\sqrt{k}}\sum_{t=1}^{k}\frac{\beta}{\mathbf{u}^\intercal \mathbf{v}}\sum_{j=1}^{n}\nabla g_j(x^*)\mathbf{T}_j\left(G_j(x^*;\phi_{j,k+1}^{'})-g_j(x^*)\right)\stackrel{d}{\rightarrow} N\left(\mathbf{0},\left(\frac{\beta}{\mathbf{u}^\intercal \mathbf{v}}\right)^2\mathbf{S}_2\right),
\end{equation*}
where
\begin{align*}
&\epsilon_t^*=\left[\left(\nabla G_1(x^*;\phi_{1,t})\nabla F_1(g(x^*);\zeta_{1,t})-\nabla g_1(x^*;\phi_{1,t})\nabla f_1(g(x^*))\right)^\intercal,\cdots,\right.\\
&\quad\quad\left.\left(\nabla G_n(x^*;\phi_{n,t})\nabla F_n(g(x^*);\zeta_{n,t})-\nabla g_n(x^*;\phi_{n,t})\nabla f_n(g(x^*))\right)^\intercal\right]^\intercal.
\end{align*}
Note that {\small$\mathbf{H}_{\theta}^{-1}=\left(
\begin{array}{cc}
n\mathbf{H}^{-1}& -\frac{\mathbf{u}^\intercal \mathbf{v}}{\beta}\mathbf{H}^{-1}\\
\mathbf{0}& \frac{\mathbf{u}^\intercal \mathbf{v}}{n\beta}\mathbf{I}_d
\end{array}
\right)$}
%
and $\phi_{i,k}$ is independent of $\phi_{i,k}^{'}$.
Then
\begin{equation*}
\frac{1}{\sqrt{k}}\sum_{t=1}^{k}\mathbf{H}_{\theta}^{-1}\eta_t^{(1)}\stackrel{d}{\longrightarrow}
N\left(\mathbf{0},\left(
\begin{array}{cc}
\mathbf{H}^{-1}\left(\mathbf{S}_1+\mathbf{S}_2\right)(\mathbf{H}^{-1})^\intercal& -\frac{1}{n}\mathbf{H}^{-1}\mathbf{S}_2\\
-\frac{1}{n}\mathbf{S}_2(\mathbf{H}^{-1})^\intercal& \frac{1}{n^2}\mathbf{S}_2
\end{array}
\right)\right).
\end{equation*}
The proof is complete.
\end{proof}

Theorem \ref{thm:asym-norm} shows that Polyak-Ruppert averaged iterates of the proposed method converge
in distribution to a normal random vector for any agent. Different from the traditional  asymptotic normality results on SA based methods \cite{chung1954stochastic,fabian1968asymptotic}, the  asymptotic covariance matrix in (\ref{limit distribution}) has two parts,  $\mathbf{H}^{-1}\mathbf{S}_1(\mathbf{H}^{-1})^\intercal$ and $\mathbf{H}^{-1}\mathbf{S}_2(\mathbf{H}^{-1})^\intercal$, where the first one is induced by the randomness of gradient and the second one is induced by the randomness of the inner function. Indeed, the    asymptotic  normality   on the SAA scheme for stochastic compositional optimization  has been studied by  Dentcheva et al. \cite{Dentcheva2017}.
To the best of our knowledge, Theorem \ref{thm:asym-norm} is the first asymptotic normality result for the SA based method on distributed  stochastic compositional optimization problem.

\section{Experimental Results} 
We test the proposed method for two applications, i.e.,   model-agnostic meta learning  problem and logistic regression problem.
\subsection{Model-agnostic meta learning}\label{maml}
Model-agnostic meta learning (MAML)  is a powerful tool for learning a new task by using the prior
experience from related tasks \cite{maml-finn-17}. It is to
find a good initialization parameter from similar learning tasks such that taking several gradient steps would produce good results on  new tasks,  and the optimizations model is
\begin{equation}\label{MAML}
\min_{x\in\mathbb{R}^d}\frac{1}{M}\sum_{m=1}^M f_m\left(x-\alpha\nabla f_m(x)\right),	
\end{equation}
where $m=1,2,\cdots,M$ is the index of training tasks, $\alpha$ is the adaptation stepsize, $f_m(x)=\mathbb{E}\left[F_m(x;\zeta_m)\right]$ is the loss function of task $m$.
We illustrate the empirical performance of AB-DSCSC to solve MAML problem (\ref{MAML}) and compare it with GP-DSCGD and GT-DSCGD \cite{gao2021fast}.
\begin{figure}[t]
	\centering
	\begin{minipage}[t]{0.41\textwidth}
		\includegraphics[height=2.4in,width=2.9in]{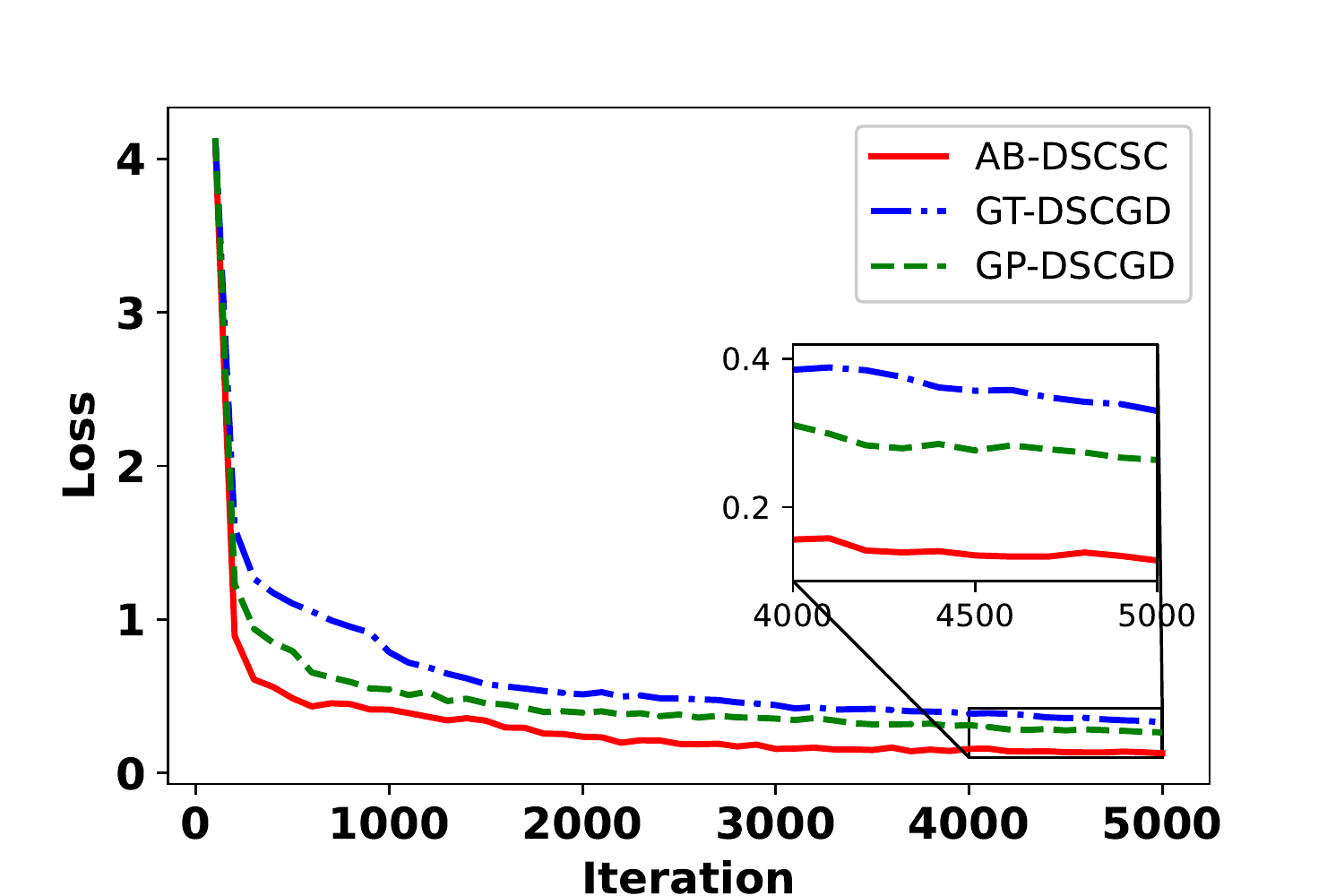}
	\end{minipage}
	\hspace{1.5cm}
	\begin{minipage}[t]{0.43\textwidth}
		\includegraphics[height=2.4in,width=2.9in]{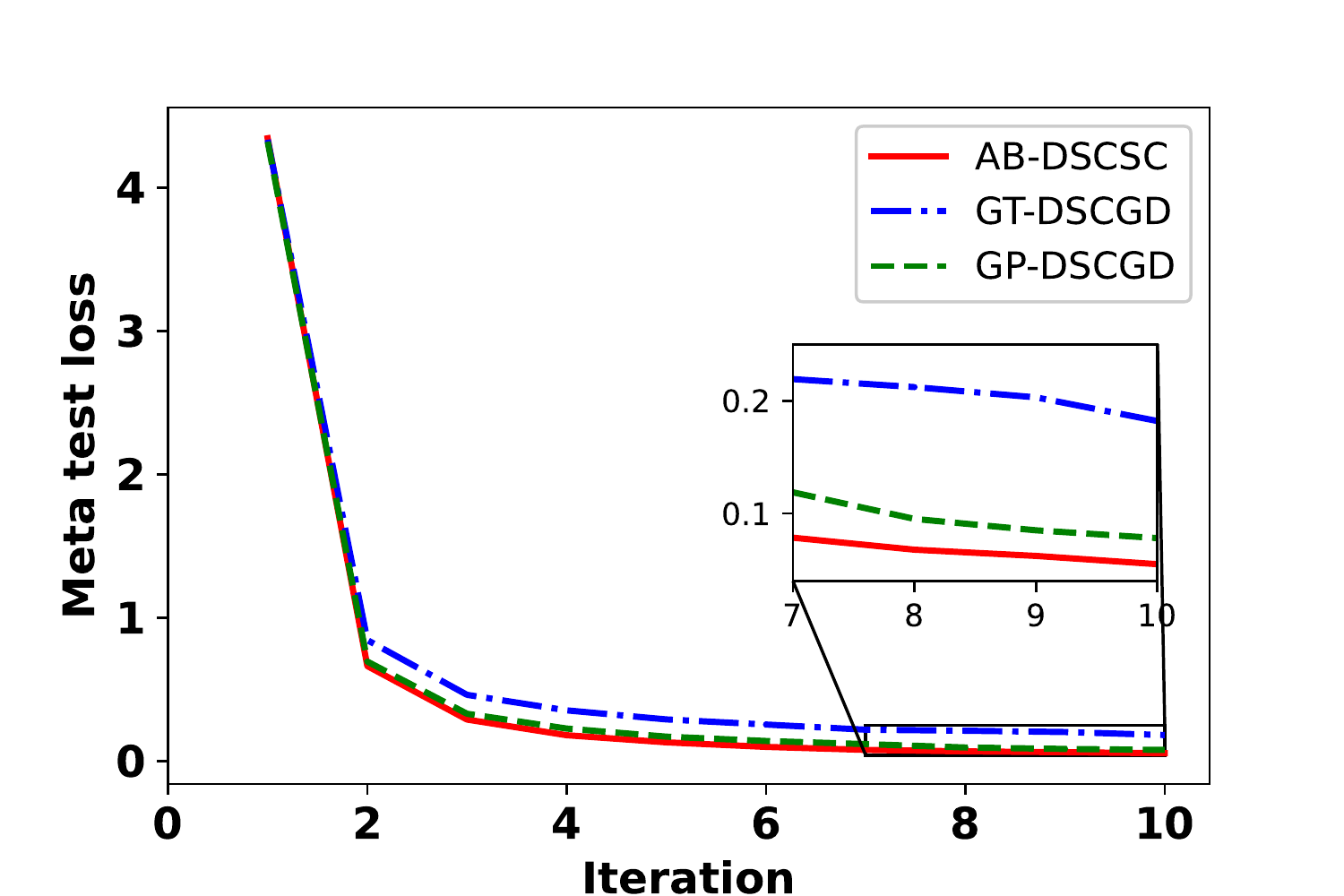}
	\end{minipage}
	\caption{Meta-training and meta-test.}\label{fig:maml}
\end{figure}

The setting of MAML is as follows \cite{gao2021fast,tianyi2021sol}.  Each task $m \in \mathcal{M}=\{1,2,\cdots,M\}$ maps the input $b$ to a sine wave $s(b;a_m,\phi_m)\define a_m\sin(b+\phi_m)$ where the amplitude $a_m$ and phase $\phi_m$ of the sinusoid vary across tasks. The tasks' parameters $a_m$ and $\phi_m$ are sampled uniformly from $[0.1,5]$ and $[0,2\pi]$ respectively, input domain of $b$ is uniform on $[5,-5]$. The regressor of $s(b;a_m,\phi_m)$ is a fully-connected neural network $\hat{s}(b;x)$, which consists of two hidden layers with 40 ReLU nodes. The loss function $f_m(z)=\mathbb{E}_{b}\left[\left\|\hat{s}(b;z)-a_m\sin(b+\phi_m)\right\|^2\right]$   and  the one-step adaptation stepsize  $\alpha=0.01$.

In this experiment, we  generate a directed graph $\mathcal{G}$ of 5 agents by adding random links to a ring network. Each agent is assigned with 200 training tasks, i.e. $M=1000$ in problem (\ref{MAML}). We utilize 2500 new tasks of sinusoidal regression to test the obtained parameters.  For AB-DSCSC, stepsize $\alpha_k=0.01,\beta_k=0.8$ and communication graphs $\mathcal{G}_\A=\mathcal{G}_{\Z^\intercal}=\mathcal{G}$. For GP-DSCGD and GT-DSCGD, stepsize $\eta=0.03,\gamma=3$, $\beta_k=0.33$,  and set the underlying graph\footnote{The underlying graph of a directed graph $\mathcal{G}^{'}$ is an undirected graph obtained by replacing all directed edges of $\mathcal{G}^{'}$ with undirected edges.} of $\mathcal{G}$ as the communication graph. In each task, we use 10 samples for training and testing.

We run AB-DSCSC, GP-DCSGD and GT-DCSGD for 5000 iterations and record their performance on the
training loss and test loss in Figure \ref{fig:maml}, where the solid curve, dash-dot curve and dashed curve display the
averaged training loss of AB-DSCSC, GT-DCSGD and GP-DCSGD over different agents respectively. We can observe from Figure \ref{fig:maml} (left) that the three methods achieve similar performance. Figure \ref{fig:maml} (right) depicts the test loss on new tasks after 10 gradient descent steps with the learned model parameters as the initialization.  Again, the three methods achieve similar performance on the new tasks and they are well adaptable to new tasks as
test loss decreasing quickly.

\subsection{Conditional stochastic optimization}

We consider a modified logistic regression problem, in which the inner and outer randomness are independent of each other\cite{Hu2020sample},
\begin{equation}\label{logre}
\min_{x\in\mathbb{R}^d} h(x)=\frac{1}{n}\sum_{i=1}^n \frac{1}{m_i-m_{i-1}}\sum_{j=m_{i-1}+1}^{m_i}\log \left(1+\exp\left(-b_j\left(\frac{1}{l}\sum_{s=1}^l\phi_{s}+a_j\right)^\intercal x\right)\right),
\end{equation}
where $n=50$, $m_i=20i$, $l=10000$, $a_j\sim N(\mathbf{0},\mathbf{I}_d)$, $b_j\in\{1,-1\}$, $\phi_{s}\sim N(\mathbf{0},\mathbf{I}_d)$.
Obviously, problem (\ref{logre}) falls in the form of DSCO with
inner function
$$g_i(x)=\left[-b_{m_{i-1}+1}\left(\frac{1}{l}\sum_{s=1}^l\phi_{s}+a_{m_{i-1}+1}\right)^\intercal x,\cdots,~-b_{m_i}\left(\frac{1}{l}\sum_{s=1}^l\phi_{s}+a_{m_i}\right)^\intercal x\right]^\intercal,$$
outer function $f_i(z)=\frac{1}{m_i-m_{i-1}}\sum_{j=m_{i-1}+1}^{m_i}\log \left(1+\exp\left(z_j\right)\right)$, $z_j$ is the $j$-th component of vector $z\in \mathbb{R}^{m_i}$.
\begin{figure}[H]
	\centering
	\begin{minipage}[t]{0.41\textwidth}
		\centering
		\includegraphics[height=2.4in,width=2.9in]{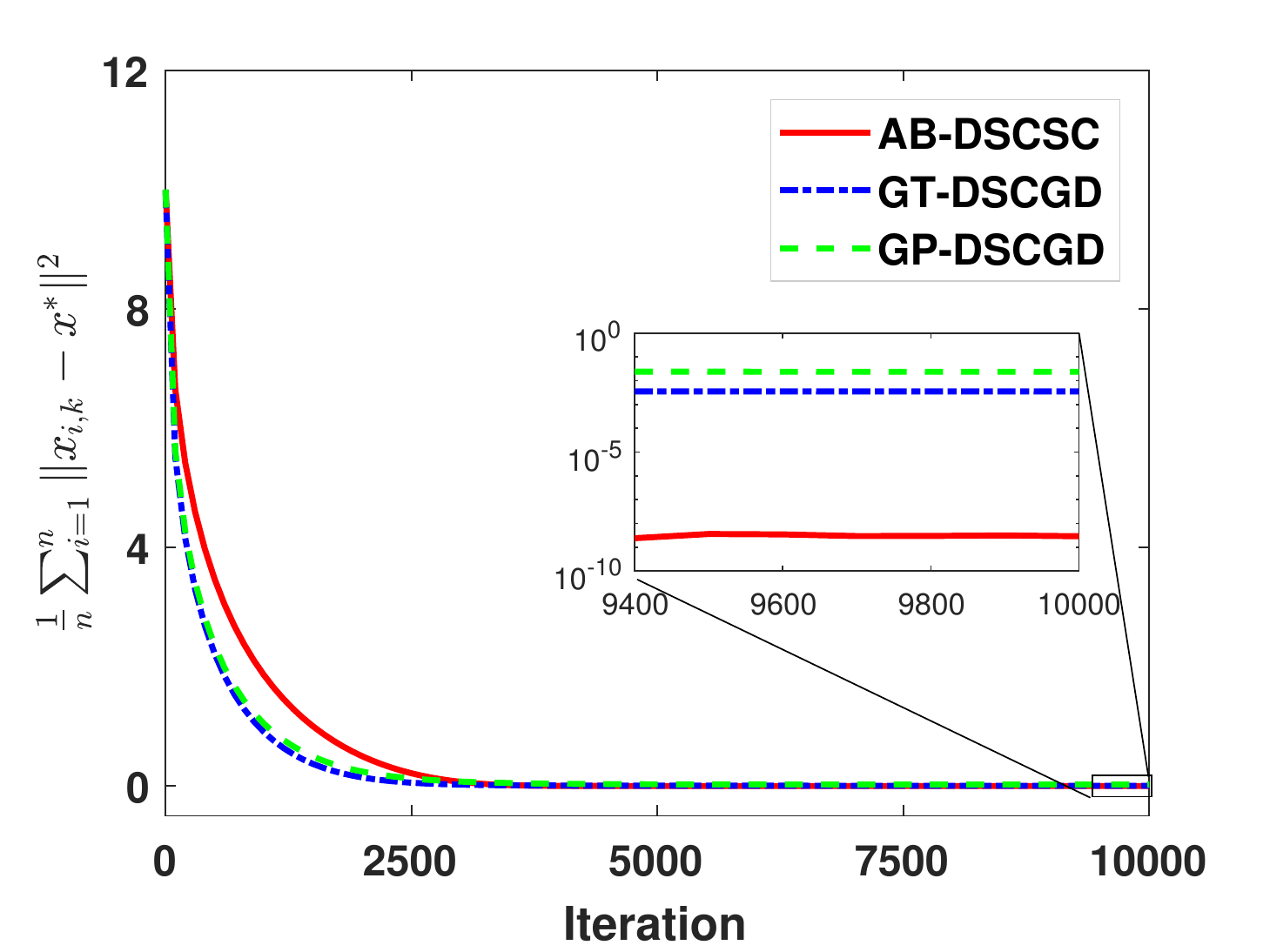}
	\end{minipage}
	\hspace{1.5cm}
	\begin{minipage}[t]{0.43\textwidth}
		\centering
		\includegraphics[height=2.4in,width=2.9in]{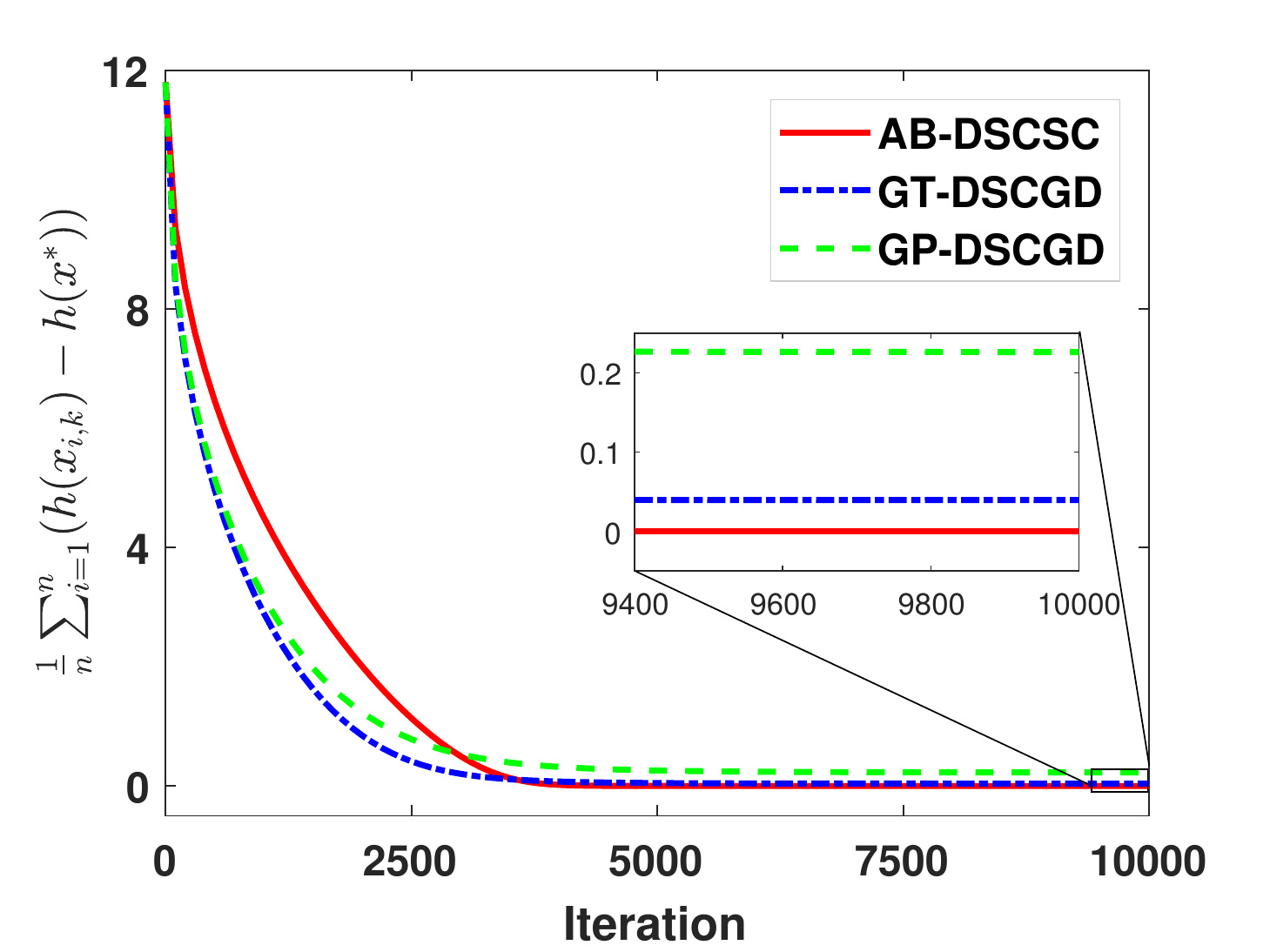}
	\end{minipage}
	\caption{Optimality gap and residual.}\label{fig:loge-re}
\end{figure}
Similarly, we  generate a directed graph $\mathcal{G}$ of 50 agents by adding random links to a ring network, and set communication graphs $\mathcal{G}_\A=\mathcal{G}_{\Z^\intercal}=\mathcal{G}$ for AB-DSCSC. The communication graph of GP-DCSGD and GT-DSCGD is also set as the underlying graph of $\mathcal{G}$.
The stepsize $\alpha_{k}=0.01/k^{0.55},~\beta_k=0.8/k^{0.6}$ for AB-DSCSC and  $\eta=0.03,\gamma=3,\beta_k=0.33/k^{0.6}$ for GP-DCSGD and GT-DCSGD.

Note that problem (\ref{logre}) is  a convex optimization problem, we solve it by centralized gradient descent and denote the optimal solution as $x^*$. Then, we run AB-DSCSC, GP-DCSGD and GT-DCSGD for $10000$ iterations and record their performance on the averaged optimality gap $\frac{1}{n}\sum_{i=1}^n\|x_{i,k}-x^*\|^2$ and average residual $\frac{1}{n}\sum_{i=1}^n (h(x_{i,k})-h(x^*))$ in Figure \ref{fig:loge-re}. Obviously, the three methods can solve the problem efficiently and achieve similar performance.

\textbf{Acknowledgment.} The authors thank  
Dr. Yuejiao Sun for sharing the code of SCSC  \cite{tianyi2021sol}. The research is supported by the NSFC \#11971090.

\bibliographystyle{siam}
\bibliography{dsi_mybib}

\section*{Appendix}
\begin{appendices}

\begin{lem}\label{lem:noi-bound}
	Let $\alpha_k=a/(k+b)^\alpha$, $a>0,b\ge0$, $\alpha\in (1/2, 1]$. Under Assumptions \ref{ass-objective}-\ref{ass:matrix} and the condition that objective function $h(x)$ is $\mu$-strongly convex,
	$$\mathbb{E}\left[\left\langle \bar{x}_{k}-x^*,-\alpha_{k}\left(\frac{\mathbf{u}^\intercal}{n}\otimes\mathbf{I}_{d}\right)\xi_k\right\rangle\right]\le \frac{\|\mathbf{u}\|c_bc_0}{2n(1-\tau_\Z)}\left(\frac{c_b^2nC_gC_f}{(1-\tau_\Z)^2}+4nC_g C_f\right)\alpha_k^2,$$
	where $c_b=\max\left\{\X,\frac{\normm{\mathbf{B}-\mathbf{I}_{n}}_\Z}{\tau_\Z}\X\right\}$, $c_0$ is some constant scalar.
	
\end{lem}
\begin{proof}
	Recall the definition $\bar{x}_{k}= \left(\frac{\mathbf{u}^\intercal}{n}\otimes\mathbf{I}_{d}\right)\mathbf{x}_{k}$ in Lemma \ref{lem:rate},
	\begin{equation*}
	\begin{aligned}
	\bar{x}_{k}-x^*
	=\bar{x}_{k-1}-x^*-\alpha_{k-1}\left(\frac{\mathbf{u}^\intercal}{n}\otimes\mathbf{I}_{d}\right)\mathbf{y}_{k-1}=\bar{x}_{1}-x^*-\sum_{t=1}^{k-1}\alpha_t\left(\frac{\mathbf{u}^\intercal}{n}\otimes\mathbf{I}_{d}\right)\mathbf{y}_t,
	\end{aligned}
	\end{equation*}
	and then
	\begin{equation*}
	\begin{aligned}
	&\mathbb{E}\left[\left\langle \bar{x}_{k}-x^*, -\alpha_{k}\left(\frac{\mathbf{u}^\intercal}{n}\otimes\mathbf{I}_{d}\right)\xi_k\right\rangle\right]\\
	&=\mathbb{E}\left[\left\langle \bar{x}_{1}-x^*-\sum_{t=1}^{k-1}\alpha_t\left(\frac{\mathbf{u}^\intercal}{n}\otimes\mathbf{I}_{d}\right)\mathbf{y}_t, -\alpha_{k}\left(\frac{\mathbf{u}^\intercal}{n}\otimes\mathbf{I}_{d}\right)\xi_k\right\rangle\right]\\
	&=-\alpha_{k}\mathbb{E}\left[\left\langle \bar{x}_{1}-x^*-\sum_{t=1}^{k-1}\alpha_t\left(\frac{\mathbf{u}^\intercal}{n}\otimes\mathbf{I}_{d}\right)\mathbf{y}_t, \left(\frac{\mathbf{u}^\intercal}{n}\otimes\mathbf{I}_{d}\right)\sum_{t=1}^{k}\tilde{\mathbf{B}}(k,t)\epsilon_t\right\rangle\right],
	\end{aligned}
	\end{equation*}
	where $\epsilon_t=\mathbf{H}_t-\mathbf{J}_t$, $\mathbf{H}_t$ and $\mathbf{J}_t$ are defined in (\ref{alg:new form}) and Lemma \ref{lem:rate} respectively, the second equality follows from (\ref{e-repre}). Note that $\mathbb{E}\left[\left\langle\bar{x}_{0}-x^*, \left(\frac{\mathbf{u}^\intercal}{n}\otimes\mathbf{I}_{d}\right)\tilde{\mathbf{B}}(k,t)\epsilon_t\right\rangle\bigg|\mathcal{F}_t\right]=0$ and 
   \begin{equation*}
	\mathbb{E}\left[\left\langle \left(\frac{\mathbf{u}^\intercal}{n}\otimes\mathbf{I}_{d}\right)\mathbf{y}_{t_1}, \left(\frac{\mathbf{u}^\intercal}{n}\otimes\mathbf{I}_{d}\right)\tilde{\mathbf{B}}(k,t_2)\epsilon_{t_2}\right\rangle\bigg|\mathcal{F}_{t_2}\right]=0~ (t_1<t_2),
	\end{equation*}	
	where $\mathcal{F}_k$ is defined in (\ref{s-algebra}). Then
	
	\begin{equation*}
	\begin{aligned}
	\mathbb{E}\left[\left\langle \bar{x}_{k}-x^*, -\alpha_{k}\left(\frac{\mathbf{u}^\intercal}{n}\otimes\mathbf{I}_{d}\right)\xi_k\right\rangle\right]
	&\le \alpha_{k}\sum_{t_1=1}^{k-1}\sum_{t_2=1}^{t_1}\alpha_{t_1}\frac{\|\mathbf{u}\|^2}{2n^2}\normm{\tilde{\mathbf{B}}(k,t_2)}\left(\mathbb{E}\left[\|\mathbf{y}_{t_1}\|^2\right]+\mathbb{E}\left[\left\| \epsilon_{t_2}\right\|^2\right]\right)\notag\\
	&\le \alpha_{k}\sum_{t_1=1}^{k-1}\sum_{t_2=1}^{t_1}\alpha_{t_1}\frac{\|\mathbf{u}\|^2c_b}{2n^2}\tau_\Z^{k-t_2}\left(\frac{c_b^2nC_gC_f}{(1-\tau_\Z)^2}+4nC_g C_f\right)\notag\\
	&\le \frac{\|\mathbf{u}\|^2c_bc}{2n^2(1-\tau_\Z)}\left(\frac{c_b^2nC_gC_f}{(1-\tau_\Z)^2}+4nC_g C_f\right)\alpha_k\alpha_{k-1},
	\end{aligned}
	\end{equation*}
	where $c_b=\max\left\{\X,\frac{\normm{\mathbf{B}-\mathbf{I}_{n}}_\Z}{\tau_\Z}\X\right\}$ and $c$ is some constant scalar,
	the second inequality follows from (\ref{B-bound}), (\ref{y-bound}) and (\ref{e-bound}),
	the third inequality follows from Lemma \ref{lem:weighted-seq}. Noting that $\lim_{k\rightarrow\infty}\frac{\alpha_{k-1}}{\alpha_k}=1$, there exists constant $c_0>c$ such that
	$$\mathbb{E}\left[\left\langle \bar{x}_{k}-x^*,-\alpha_{k}\left(\frac{\mathbf{u}^\intercal}{n}\otimes\mathbf{I}_{d}\right)\xi_k\right\rangle\right]\le \frac{\|\mathbf{u}\|^2c_bc_0}{2n^2(1-\tau_\Z)}\left(\frac{c_b^2nC_gC_f}{(1-\tau_\Z)^2}+4nC_g C_f\right)\alpha_k^2.$$
	The proof is complete.

\end{proof}

\begin{lem}\label{lem:con-pro}
	Let $\alpha_{k}=a/(k+b)^\alpha$, $a>0$, $b\ge 0$, $\alpha\in(1/2,1)$.
	Suppose that
	\begin{itemize}
		\item[(a)] Assumptions \ref{ass-objective}-\ref{ass:matrix} hold;
		
		\item [(b)] for any $i\in\mathcal{V}$, there exist scalar $C_i$ and matrix  $\mathbf{T}_i$ such that
		\begin{equation*}
		\left\|\nabla f_i(y)-\nabla f_i(y^{'})-\mathbf{T}_i\left(y-y^{'}\right)\right\|\le C_i\|y-y^{'}\|^{1+\gamma},\quad \forall y,y^{'}\in \mathbb{R}^p, 
		\end{equation*}
		where $\gamma\in (0,1]$ satisfies that $\sum_{k=1}^\infty\frac{\alpha_k^{(1+\gamma)/2}}{\sqrt{k}}<\infty$.
	\end{itemize}
Denote  $\mathbf{M}(k,t)=\tilde{\alpha}_t\sum_{l_1=t}^k\Pi_{l_2=t+1}^{l_1}\left(\mathbf{I}_{2d}-\tilde{\alpha}_k\mathbf{H}_{\theta}\right)
,~\mathbf{N}(k,t)=\mathbf{M}(k,t)-\mathbf{H}_{\theta}^{-1}
$ and
\begin{equation*}
\eta_t^{(1)}=\left(
\begin{array}{c}
\left(-\frac{n}{\mathbf{u}^\intercal \mathbf{v}}\right)\left(\frac{\mathbf{u}^\intercal}{n}\otimes\mathbf{I}_{d}\right)\xi_t\\
\frac{\beta}{\mathbf{u}^\intercal \mathbf{v}}\sum_{j=1}^{n}\nabla g_j(x^*)\mathbf{T}_j\left(G_j(x^*;\phi_{i,t+1}^{'})-g_j(x^*)\right)
\end{array}
\right).
\end{equation*} 
We have
	\begin{equation*}
	\lim_{k\rightarrow\infty}\mathbb{E}\left[\left\|\frac{1}{\sqrt{k}}\sum_{t=1}^{k}\mathbf{N}(k,t)\eta_t^{(1)}\right\|^2\right]=0.
	\end{equation*}
\end{lem}
\begin{proof}
	Note that 
	\begin{equation*}
	\eta_t^{(1)}=\left(
	\begin{array}{c}
	\left(-\frac{n}{\mathbf{u}^\intercal \mathbf{v}}\right)\left(\frac{\mathbf{u}^\intercal}{n}\otimes\mathbf{I}_{d}\right)\xi_t\\
	\mathbf{0}
	\end{array}
	\right)+\left(
	\begin{array}{c}
	\mathbf{0}\\
	\frac{\beta}{\mathbf{u}^\intercal \mathbf{v}}\sum_{j=1}^{n}\nabla g_j(x^*)\mathbf{T}_j\left(G_j(x^*;\phi_{i,t+1}^{'})-g_j(x^*)\right)
	\end{array}
	\right)
	\end{equation*}
	and 
	\begin{align*}
	&\mathbb{E}\left[\left\langle\xi_{t_1},\xi_{t_2}\right\rangle\right]=\mathbb{E}\left[\mathbb{E}\left[\left\langle\xi_{t_1},\xi_{t_2}\right\rangle\big|\mathcal{F}_{\min\{t_1,t_2\}}\right]\right]=\mathbb{E}\left[\left\langle\xi_{\min\{t_1,t_2\}},\sum_{l=1}^{\min\{t_1,t_2\}}\tilde{\mathbf{B}}(\max\{t_1,t_2\},l)\epsilon_l\right\rangle\right]~(t_1\le t_2),\\
	&\mathbb{E}\left[\left\langle G_j(x^*;\phi_{i,t_1}^{'})-g_j(x^*),G_j(x^*;\phi_{i,t_2}^{'})-g_j(x^*)\right\rangle\big|\mathcal{F}_{\min\{t_1,t_2\}}\right]=0 ~(t_1\ne t_2),
	\end{align*}
	where $\mathcal{F}_t$ is defined in (\ref{s-algebra}).
	Then
	{\small\begin{equation*}
	\begin{aligned}
	\mathbb{E}\left[\left\|\frac{1}{\sqrt{k}}\sum_{t=1}^{k}\mathbf{N}(k,t)\eta_t^{(1)}\right\|^2\right]
	&=\mathbb{E}\left[\mathbb{E}\left[\left\|\frac{1}{\sqrt{k}}\sum_{t=1}^{k}\mathbf{N}(k,t)\eta_t^{(1)}\right\|^2\bigg|\mathcal{F}_{\min\{t_1,t_2\}}\right]\right]\\
	&\le \left(\frac{\|\mathbf{u}\|}{\mathbf{u}^\intercal \mathbf{v}}\right)^2\frac{4}{k}\sum_{t_1=1}^{k}\sum_{t_2=t_1}^{k}\normm{\mathbf{N}(k,t_1)}\normm{\mathbf{N}(k,t_2)}\sum_{l=1}^{t_1}\normm{\tilde{\mathbf{B}}(t_2,l)}\mathbb{E}\left[\left\|\xi_{t_1}\right\|\left\|\epsilon_l\right\|\right]\\
	&\quad+\frac{2}{k}\sum_{t=1}^{k}\normm{\mathbf{N}(k,t)}^2\mathbb{E}\left[\left\|\frac{\beta}{\mathbf{u}^\intercal \mathbf{v}}\sum_{j=1}^{n}\nabla g_j(x^*)\mathbf{T}_j\left(G_j(x^*;\phi_{i,t}^{'})-g_j(x^*)\right)\right\|^2\right]\\
	&\le c_bc_N(c_b^2+1)nC_fC_g\left(\frac{\|\mathbf{u}\|}{\mathbf{u}^\intercal \mathbf{v}}\right)^2\frac{1}{(1-\tau_\Z)^4}\frac{8}{k}\sum_{t_1=1}^{k}\normm{\mathbf{N}(k,t_1)}\\
	&\quad+n\left(\frac{\beta }{\mathbf{u}^\intercal \mathbf{v}}\right)^2\left(\sum_{j=1}^{n}\|\nabla g_j(x^*)\|^2\|\mathbf{T}_j\|^2\right)V_gc_N\frac{2}{k}\sum_{t=1}^{k}\normm{\mathbf{N}(k,t)},
	\end{aligned}
	\end{equation*}}
	where $c_b=\max\left\{\X,\frac{\normm{\mathbf{B}-\mathbf{I}_{n}}_\Z}{\tau_\Z}\X\right\},~c_N=\sup_{k,t}\normm{\mathbf{N}(k,t)}$, the second inequality follows from the fact $\sup_{k,t}\normm{\mathbf{N}(k,t)}<\infty$ \cite[Lemma 1 (ii)]{Polyak1992}, (\ref{B-bound}), (\ref{e-bound}), Lemma  \ref{lem:nonconv} (i) and Assumption \ref{ass-objective} (c). By \cite[Lemma 1 (ii)]{Polyak1992}, $$\lim_{k\rightarrow\infty}\frac{1}{k}\sum_{t=1}^{k}\normm{\mathbf{N}(k,t)}=0,$$
	which implies $	\lim_{k\rightarrow\infty}\mathbb{E}\left[\left\|\frac{1}{\sqrt{k}}\sum_{t=1}^{k}\mathbf{N}(k,t)\eta_t^{(1)}\right\|^2\right]=0$. The proof is complete.
\end{proof}

\end{appendices}

\end{document}